%% file: self_referential_leading_digits.tex
\documentclass[11pt,a4paper]{article}
\input{preamble_en.tex}

\title{Self-Referential Leading Digits of Exponential Sequences:\\
Arithmetic Structure and Certified Search}
\author{Zihang Fang}
\newcommand{\paperaffiliation}{School of Mathematical Sciences\\
University of Science and Technology of China\\
No. 96 Jinzhai Road, Baohe District, Hefei 230026, Anhui,
People's Republic of China}
\newcommand{\paperemail}{jiuxun@mail.ustc.edu.cn}
\newcommand{\paperorcid}{0009-0001-6001-8638}
\newcommand{\paperclass}{Primary 11A63; Secondary 11J70, 11J71, 11Y16}
\newcommand{\paperkeywords}{leading digits, shrinking targets, Lambert
$W$ function, discrepancy, resonance, continued fractions, certified search}

\hypersetup{
  pdftitle={Self-Referential Leading Digits of Exponential Sequences},
  pdfauthor={Zihang Fang},
  pdfsubject={Arithmetic structure and certified search},
  pdfkeywords={leading digits, shrinking targets, Lambert W function,
    discrepancy, resonance, continued fractions, certified search}
}

\begin{document}
\maketitle

\begin{abstract}
For $c>1$ and an integer radix $b\ge2$, we study the positive integers $m$
for which $mb^k\le c^m<(m+1)b^k$ for some $k\ge0$; for integer $c$, this is
the self-prefix leading-digit condition.  We derive an exact shrinking-target
criterion; for $c\ge2$, an exact signed-discrepancy identity isolates both
infinitude and the conjectural logarithmic count.  For $c\ge2$ with
nonintegral logarithmic slope, Lambert
$W_{-1}$ inversion produces a candidate sequence with an eventual two-gap
law and an exact counting formula; for $(c,b)=(2,10)$ all consecutive
candidate gaps are $3$ or $4$.  For algebraic $c$ with irrational
$\log_b c$, the Lambert-root phases satisfy deterministic moving-target
asymptotics in an explicit nontrivial power range strictly below the critical
scale.  For irrational logarithmic slope,
actual hits obey fixed-difference and arithmetic-chain rigidity; for
multiplicatively independent integer parameters, coherent endpoint hits at
floor resonance centers force every intermediate term.  Finally, set
$\rho=\{\log_b c\}$.  For fixed multiplicatively independent integers
$c,b$, an interpolated continued-fraction locator has bit complexity
$O(N^{1-1/\nu}\operatorname{polylog}N)$ for every $\nu>\mu(\rho)$.  We give
an explicit certified instance for $(2,10)$, whose infinitude remains open.
\end{abstract}

\begin{center}
\begin{minipage}{0.90\textwidth}
\small
\textbf{2020 Mathematics Subject Classification.} \paperclass\par
\end{minipage}
\end{center}

\clearpage
\tableofcontents
\clearpage

\input{sections/en/01_introduction.tex}
\input{sections/en/02_exact_criterion.tex}
\input{sections/en/03_lambert_layers.tex}
\input{sections/en/04_counting_identity.tex}
\input{sections/en/05_arithmetic_structure.tex}
\input{sections/en/06_certified_search.tex}
\input{sections/en/07_fixed_parameter_problems.tex}
\input{sections/en/08_submission_declarations.tex}

\printbibliography[heading=bibintoc]
\end{document}

%% file: preamble_en.tex
\usepackage[T1]{fontenc}
\usepackage{newtxtext}
\usepackage[a4paper,top=24mm,bottom=24mm,left=25mm,right=25mm,
  footskip=10mm]{geometry}
\usepackage{microtype}
\usepackage{mathtools}
\usepackage{amsthm}
\usepackage{newtxmath}
\usepackage{mathrsfs}
\usepackage{booktabs}
\usepackage{tabularx}
\usepackage{float}
\usepackage{enumitem}
\usepackage{titlesec}
\usepackage{tocloft}
\usepackage{xurl}
\usepackage[hidelinks]{hyperref}
\usepackage[nameinlink,noabbrev]{cleveref}
\usepackage[backend=biber,style=numeric-comp,sorting=nyt,maxnames=99,
  giveninits=true,doi=true,url=true,isbn=false]{biblatex}
\AtBeginBibliography{%
  \small
  \Urlmuskip=0mu plus 2mu\relax
  \setlength{\bibitemsep}{0.20\baselineskip}%
  \setlength{\bibparsep}{0pt}%
}

\allowdisplaybreaks[2]
\titleformat{\section}
  {\large\bfseries}{\thesection}{0.7em}{}
\titlespacing*{\section}{0pt}{3.0ex plus 0.8ex minus 0.3ex}{1.35ex}
\titleformat{\subsection}
  {\normalsize\normalfont}{\thesubsection}{0.65em}{}
\titlespacing*{\subsection}{0pt}{2.4ex plus 0.6ex minus 0.2ex}{0.9ex}

\newtheoremstyle{arxivplain}
  {8pt}{8pt}{\itshape}{}{\bfseries}{.}{0.55em}{}
\newtheoremstyle{arxivdefinition}
  {8pt}{8pt}{\normalfont}{}{\bfseries}{.}{0.55em}{}
\newtheoremstyle{arxivremark}
  {8pt}{8pt}{\normalfont}{}{\itshape}{.}{0.55em}{}
\theoremstyle{arxivplain}
\newtheorem{theorem}{Theorem}[section]
\newtheorem{proposition}[theorem]{Proposition}
\newtheorem{corollary}[theorem]{Corollary}
\newtheorem{lemma}[theorem]{Lemma}
\theoremstyle{arxivdefinition}
\newtheorem{definition}[theorem]{Definition}
\theoremstyle{arxivremark}
\newtheorem{remark}[theorem]{Remark}
\newtheorem{problem}[theorem]{Open problem}

\renewenvironment{abstract}{%
  \par\begin{center}\bfseries Abstract\end{center}%
  \small\begin{quote}
}{%
  \end{quote}\par\vspace{0.7em}
}

\newcommand{\T}{\mathbb T}

\newcommand{\e}{\mathrm e}
\newcommand{\ind}{\mathbf 1}
\newcommand{\Sset}{\mathcal S}
\newcommand{\dist}{\operatorname{dist}}

\makeatletter
\renewcommand{\maketitle}{%
  \thispagestyle{plain}
  \begin{center}
    \vspace*{1.0cm}
    {\fontsize{18}{22}\selectfont\bfseries\@title\par}
    \vspace{1.1em}
    {\large\@author\par}
    \vspace{0.45em}
    {\small\paperaffiliation\par}
    \vspace{0.2em}
    {\small
      \textit{Corresponding author: }
      \href{mailto:\paperemail}{\paperemail}\par
      \href{https://orcid.org/\paperorcid}{ORCID: \paperorcid}\par}
  \end{center}
  \vspace{1.0em}
}
\makeatother

%% file: sections/en/01_introduction.tex
\section{Introduction}

Let $b\ge2$ be an integer radix and let $c>1$.  Define
\begin{equation}\label{eq:def-S}
  \Sset_{c,b}
  =\left\{m\ge1:\text{there is }k\in\mathbb Z_{\ge0}\text{ such that }
    mb^k\le c^m<(m+1)b^k\right\},
\end{equation}
and write $A_{c,b}(X)=\#(\Sset_{c,b}\cap[1,X])$.  When $c$ is an integer,
membership in \cref{eq:def-S} says exactly that the base-$b$ digit expansion
of $c^m$ begins with the base-$b$ digit expansion of $m$.  The motivating
case is
\begin{equation}\label{eq:core-case}
  \Sset_{2,10}=\{m\ge1:2^m\text{ begins with }m\text{ in decimal}\}.
\end{equation}
It was investigated by van de Lune~\cite{vandelune1978} and is recorded as
OEIS A100129~\cite{oeisA100129}.  Computations suggest that
$\Sset_{2,10}$ is infinite and that
\begin{equation}\label{eq:conjectural-log-law-intro}
  A_{2,10}(X)\sim\log_{10}X,
\end{equation}
but neither assertion is known.

Put
\begin{equation}\label{eq:basic-notation}
  \alpha=\log_b c,\qquad
  F_{c,b}(m)=m\alpha-\log_bm,\qquad
  \delta_b(m)=\log_b\left(1+\frac1m\right).
\end{equation}
We write $\{x\}=x-\lfloor x\rfloor$,
$d_+(x)=\lceil x\rceil-x$, and
$\|x\|=\dist(x,\mathbb Z)$; throughout, $\log$ denotes the natural
logarithm.  The elementary but decisive reformulation is
\begin{equation}\label{eq:shrinking-target-intro}
  m\in\Sset_{c,b}
  \quad\Longleftrightarrow\quad
  F_{c,b}(m)\ge0\ \text{ and }\
  \{F_{c,b}(m)\}<\delta_b(m).
\end{equation}
Thus the original problem is a one-sided shrinking-target problem at the
critical scale $\delta_b(m)\asymp1/m$.  Uniform distribution for fixed
intervals does not control this scale.  Likewise, a metric theorem holding
for almost every value of $c$ cannot be specialized to the prescribed
constant $c=2$.  We therefore study exact structure and certified
localization for fixed parameters, keeping every conclusion logically
separate from the still-open infinitude problem.

For $c\ge2$, there is also an exact reason that the expected order of
magnitude is logarithmic.  The target masses telescope:
\begin{equation}\label{eq:mass-telescope-intro}
  \sum_{m\le N}\delta_b(m)=\log_b(N+1).
\end{equation}
This identity alone does not give a counting theorem, because the phases and
the shrinking windows are generated by the same fixed orbit.  Section~4
isolates the resulting dependence exactly.  If
$y_n=n\alpha-\log_b(n+1)$ and
\[
  B_{\alpha,b}(N)
  =\#\{0\le n<N:\{y_n\}\in[1-\rho,1)\},
  \qquad \rho=\{\alpha\},
\]
then
\begin{equation}\label{eq:exact-discrepancy-intro}
  A_{c,b}(N)
  =\log_b(N+1)+B_{\alpha,b}(N)-\rho N+\{y_N\}.
\end{equation}
Thus \cref{eq:conjectural-log-law-intro} is equivalent to a signed
discrepancy estimate of order $o(\log N)$, far sharper than ordinary uniform
distribution.  We use \cref{eq:exact-discrepancy-intro} as a diagnostic
identity rather than as an unproved independence heuristic.

\subsection{Main results}

The results fall into four groups.  The first describes all possible
Lambert candidates, the second gives an exact counting identity for the
actual solutions, the third imposes arithmetic rigidity on actual hits,
and the fourth turns continued-fraction separation into a certified
sublinear search.

\emph{I. Lambert candidate geometry.}
Assume here that $c\ge2$.  Write $\alpha=s+\rho$, where
$s=\lfloor\alpha\rfloor$ and
$0<\rho<1$, and put $L=\rho\log b$.  In each sufficiently large
logarithmic layer $j$, the set of possible integers is
$\mathbb Z\cap[x_j,z_j)$, where the exact endpoints are
\begin{equation}\label{eq:lambert-endpoints-intro}
  x_j=-\frac1L W_{-1}(-Lb^{-j}),
  \qquad
  z_j=-\frac1L W_{-1}\left(-\frac{L}{b^\rho b^j}\right)-1.
\end{equation}
The width satisfies
\begin{equation}\label{eq:lambert-width-intro}
  z_j-x_j=\frac1{j\log b}
  +O_{c,b}\left(\frac{\log j}{j^2}\right),
\end{equation}
has a convergent Lagrange-inversion expansion to all orders, and is a
completely monotone function of the continuous layer parameter; see
\cref{thm:lambert-layer,prop:all-orders-width,prop:width-complete-monotonicity}.
In particular, a large layer contains at
most one integer candidate $m_j=\lceil x_j\rceil$.

The candidate sequence itself has a rigid discrete geometry.  With
$a=1/\rho$ and $k=\lfloor a\rfloor$, its consecutive gaps eventually
belong to $\{k,k+1\}$, and the longer gaps have an exact telescoping count
with a logarithmic drift.  More directly, for a suitable initial layer
$j_0$ and every integer $M>1/L$,
\begin{equation}\label{eq:candidate-count-intro}
  \#\{j\ge j_0:m_j\le M\}
  =\max\left\{0,
  \left\lfloor\rho M-\log_bM\right\rfloor-j_0+1\right\}.
\end{equation}
This is an exact count of possible candidates, not of actual prefix hits.
For $(c,b)=(2,10)$ the result is global from the first layer:
\begin{equation}\label{eq:two-candidate-summary-intro}
  m_{j+1}-m_j\in\{3,4\},
  \qquad
  \#\{j\ge1:m_j\le M\}
  =\left\lfloor M\log_{10}2-\log_{10}M\right\rfloor
  \quad(M\ge2).
\end{equation}
These statements are proved in
\cref{thm:lambert-candidate-two-gap,cor:two-candidate-two-gap}.
They include more than a qualitative two-gap assertion.  If
$\vartheta=\{1/\rho\}$ and $G_+(N)$ counts the longer gaps from a fixed
large layer to $N$, then
\begin{equation}\label{eq:long-gap-drift-intro}
  G_+(N)=\vartheta N+\frac{\log N}{L}+O_{c,b}(1).
\end{equation}
After subtracting the linear term and this deterministic logarithmic drift,
the count on every sufficiently late interval has absolute discrepancy
smaller than $1+o(1)$.  Even when $1/\rho$ is an integer, the longer gap
therefore occurs infinitely often, but only logarithmically often.  This
candidate geometry is independent of the harder question of which $m_j$
also satisfy $m_j<z_j$.

Lambert inversion also permits a deterministic distribution theorem for
fixed algebraic parameters.  If $c$ is algebraic and $\log_b c$ is
irrational, an effective finite-type exponent $\eta\ge1$ is obtained from
Matveev's lower bound for linear forms in logarithms~\cite{matveev2000}.
Combining it with the Tichy--Turnwald discrepancy theorem
\cite{tichyturnwald1986}, we prove that whenever
\begin{equation}\label{eq:subcritical-range-intro}
  0<\tau<\sigma<\frac1{\eta+1}\le\frac12,
\end{equation}
one has, for every fixed $\kappa>0$,
\begin{equation}\label{eq:subcritical-asymptotic-intro}
  \#\{j\le N:d_+(x_j)<\kappa j^{-\tau}\}
  =\frac{\kappa}{1-\tau}N^{1-\tau}
   +O_{c,b,\kappa,\tau,\sigma}
    \left(N^{1-(\sigma+\tau)/2}\right).
\end{equation}
The estimate is deterministic for the prescribed algebraic pair.  Its
range is nevertheless strictly subcritical: the genuine prefix width in
\cref{eq:lambert-width-intro} has exponent $\tau=1$.

\emph{II. Exact counting and signed discrepancy.}
The telescoping target mass in \cref{eq:mass-telescope-intro} is not merely
heuristic.  The exact identity \cref{eq:exact-discrepancy-intro}, proved as
\cref{thm:exact-discrepancy}, converts the original count into the signed
discrepancy of one fixed interval for the nonlinear sequence
$y_n=n\alpha-\log_b(n+1)$.  In particular, with
$D_{\alpha,b}(N)=B_{\alpha,b}(N)-\rho N$,
\begin{align}
  \Sset_{c,b}\text{ is infinite}
  &\Longleftrightarrow
  \log_b(N+1)+D_{\alpha,b}(N)\text{ is unbounded},
  \label{eq:infinitude-discrepancy-summary-intro}\\
  A_{c,b}(N)\sim\log_bN
  &\Longleftrightarrow D_{\alpha,b}(N)=o(\log N).
  \label{eq:loglaw-discrepancy-summary-intro}
\end{align}
Ordinary equidistribution only gives $D_{\alpha,b}(N)=o(N)$, so the
identity identifies the precise fixed-parameter obstruction instead of
assuming independence between the phases and the windows.

\emph{III. Arithmetic resonances of actual hits.}
For multiplicatively dependent integers $c=d^p$ and $b=d^q$, we completely
classify the solution set by an eventually periodic congruence and obtain
\begin{equation}\label{eq:dependent-summary-intro}
  A_{c,b}(X)=\log_bX+O_{c,b}(1);
\end{equation}
see \cref{thm:dependent}.  The remaining results in this group concern
multiplicatively independent parameters and therefore do not follow from
that periodic mechanism.

Two different notions of resonance are needed.  The first starts with two
unknown hits separated by a fixed difference and derives a necessary shell
containing their smaller index.  It therefore controls every repeated
difference.  The second starts with a good rational approximation
$p/q<\alpha$, constructs a distinguished floor center, and proves exact
containment among the corresponding prefix windows.  It gives stronger
conclusions on a thinner, explicitly organized family.  Keeping these two
mechanisms separate avoids treating a sufficient nesting configuration as
if it classified all double hits.

Every pair of actual hits with a prescribed difference $h$ lies in one of
finitely many explicit resonance shells.  This yields effective fixed-gap
sparsity and an exact count of the candidate integers in each shell, hence
an explicit upper sieve for double hits.  If
$\mu(\alpha)<\infty$, then, for every $\nu>\mu(\alpha)$,
\begin{equation}\label{eq:fixed-difference-summary-intro}
  r_{c,b}(h)
  :=\#\{m:m,m+h\in\Sset_{c,b}\}
  \ll_{c,b,\nu}h^{\nu-2}.
\end{equation}
The same shell geometry controls arithmetic progressions of hits.  Writing
$\beta_q=\{q\alpha\}$, their maximal length satisfies
\begin{equation}\label{eq:chain-bound-summary-intro}
  \Lambda_{c,b}(q)
  \le2+\left\lceil
  \frac{b^{\beta_q}+1}{q(b^{\beta_q}-1)}
  \right\rceil.
\end{equation}
Moreover, once a chain is sufficiently long in terms of $b$, the rational
$\lfloor q\alpha\rfloor/q$, after reduction to lowest terms, must be a
regular continued-fraction convergent to $\alpha$ from below.  The
reduction clause is essential: $q$ itself need not be a convergent
denominator.  These results appear in
\cref{thm:resonance-shell,thm:arithmetic-hit-chain,cor:long-chain-convergent-direction,prop:fixed-difference-exact-sieve}.

For multiplicatively independent integers $c,b\ge2$, a second rigidity
mechanism occurs at floor resonance centers.  Given
$q\alpha-p=\varepsilon>0$, put
\begin{equation}\label{eq:floor-center-summary-intro}
  \varrho=b^\varepsilon,
  \qquad M=\left\lfloor\frac q{\varrho-1}\right\rfloor.
\end{equation}
When $M\ge1$, the later prefix window at this distinguished integer is
contained in the earlier one, so $M+q\in\Sset_{c,b}$ implies
$M\in\Sset_{c,b}$.  For the full family of shifted windows, their common
coherent intersection persists through a maximal finite range given exactly
by
\begin{equation}\label{eq:coherent-depth-summary-intro}
  J_{\max}
  =\max\{J\ge1:\varrho^{J-1}(M+q)<M+Jq+1\}.
\end{equation}
The strict inequality records the half-open upper endpoint and is essential.
Within that range, hits at the two endpoints of
\begin{equation}\label{eq:endpoint-filling-summary-intro}
  M+q,\ M+2q,\ldots,\ M+Jq
\end{equation}
force $M$ and every intermediate term of this progression to be hits, and
their radix scales form an arithmetic progression as well.  The exact
coherent depth has a closed
form involving $W_{-1}$.  After finitely many initial centers, positive-error
continued-fraction convergents give centers growing by at least a factor of
four.  For each such convergent $p_j/q_j$, write
$\varepsilon_j=q_j\alpha-p_j>0$, and let $M_j$ and $J_j$ be its floor center
and exact coherent depth.  Set
\begin{equation}\label{eq:coherent-skeleton-definition-intro}
  T(X)=\sum_{\substack{\varepsilon_j>0\\M_j\le X}}(J_j+1).
\end{equation}
Thus $T(X)$ is the aggregate size of the full potential skeletons whose
centers are at most $X$; some skeleton points may themselves exceed $X$.
For every $\epsilon>0$,
\begin{equation}\label{eq:coherent-skeleton-summary-intro}
  T(X)=o_{c,b}(\sqrt X),
  \qquad
  T(X)\ll_{c,b,\epsilon}
  X^{1/2-1/\mu(\alpha)+\epsilon}
  \quad\text{when }\mu(\alpha)<\infty.
\end{equation}
See \cref{thm:floor-resonance-nesting,thm:multistep-floor-resonance,prop:exact-coherent-depth,thm:convergent-center-skeleton}.
These are
rigidity statements about potential configurations: they do not assert
that any resonance center is itself a hit.

\emph{IV. Certified sublinear localization.}
Continued fractions also separate the phases inside a block of consecutive
indices.  We turn this separation into an exact locator using outward
interval arithmetic, rational grid phases, Euclidean floor sums, and final
integer verification.  The output-sensitive form returns, without scanning
the block, a certified superset containing every genuine hit; see
\cref{thm:safe-locator,prop:output-sensitive-locator}.

At a high level, one block is processed as follows.  Adjacent convergent
denominators provide a lower bound for the circular separation of its
phases.  The true target arc is enclosed outward on a rational grid, with a
fixed reserve smaller than that separation.  Euclidean floor sums count the
grid phases in an index interval, and binary descent locates the unique
nonempty subinterval when a reported index exists.  Only that index is
then subjected to the original inequalities in \cref{eq:def-S}.  Certified
ball arithmetic and a linear-form separation bound ensure that all floors,
ceilings, and endpoint signs are decided exactly rather than in a real-RAM
model.

The interpolated construction chooses, at each left endpoint, the largest
half-margin block supplied by all available convergents.  Let
$\rho=\{\log_b c\}$ and suppose that $\mu(\rho)<\infty$.  For every
$\nu>\mu(\rho)$, the number of certified blocks required up to $N$ is
\begin{equation}\label{eq:block-count-summary-intro}
  O_{\rho,b,\nu}\left(N^{1-1/\nu}\right).
\end{equation}
For fixed multiplicatively independent integers $c,b$, including
construction of the convergents, grid endpoints, floor sums, and exact
verification of every reported index, the total bit complexity is
\begin{equation}\label{eq:bit-complexity-summary-intro}
  O_{c,b,\nu}\left(
  N^{1-1/\nu}\operatorname{polylog}N\right).
\end{equation}
If the partial quotients are bounded, the exponent is $1/2$.  This is
\cref{thm:interpolated-blocks}.  For $(2,10)$ we additionally give a
machine-checkable interval-arithmetic certificate showing that, beyond
\begin{equation}\label{eq:explicit-threshold-summary-intro}
  1\,914\,818\,931\,502\,442,
\end{equation}
every block of length at most $44\,699\,994$ reduces to at most one
possible solution index.  The fixed block certificate is a concrete instance of the
locator, whereas the growing-denominator interpolation is what yields the
global sublinear bound.  Conversely, a singleton report is not a positive
hit certificate: it only reduces the block to one index requiring exact
verification.  This distinction is why extending the numerical range, by
itself, cannot prove either finiteness or infinitude.

\subsection{Relation to previous work}

The classical connection between leading digits and uniform distribution
modulo one is a basic source of the present problem; see
\cite{diaconis1977,bergerhill2015}.  Local Benford laws and the combinatorial
structure of leading-digit sequences give finer versions of that picture
\cite{caihildebrandli2019,heetal2020,rajagopaletal1984}.  The event studied
here is nevertheless of a different, diagonal form: at time $m$ the required
leading string is the expansion of $m$ itself.  Equivalently, the relevant
arc has endpoints $\log_bm$ and $\log_b(m+1)$ modulo one, so both its position
and its length depend on $m$, with length asymptotic to $(m\log b)^{-1}$.
Our concern is the exact arithmetic structure and certified localization of
this self-referential event for a prescribed orbit, rather than a general
leading-digit distribution law.

Shrinking targets for circle rotations and inhomogeneous approximation have
a substantial metric and dynamical theory
\cite{kurzweil1955,kim2007shrinking,tseng2008,kim2014refined,fuchskim2016};
recent work also treats moving targets under other hypotheses
\cite{michaudramirez2026}.  These results provide the natural context for
\cref{eq:shrinking-target-intro}, but their quantifiers and target
assumptions differ from the fixed-parameter question here.  Indeed, our
orbit may be written either as the nonlinear phase
$\{m\alpha-\log_bm\}$ hitting $[0,\delta_b(m))$, or as the rotation
$\{m\alpha\}$ hitting a prescribed moving arc.  Thus an almost-everywhere
recurrence statement, or a theorem for a different class of moving targets,
does not by specialization decide the critical recurrence problem for the
specified value $\alpha=\log_{10}2$.  We do not claim a general
non-applicability result for the shrinking-target literature.

One-sided Diophantine approximation is studied systematically in
\cite{hanclturek2019}.  Our use of it is more specific.  A sufficiently long
arithmetic progression of actual prefix hits first forces an explicit
one-sided inequality
\[
  0<\alpha-\frac pq<\frac1{2q^2}.
\]
After $p/q$ is reduced, the conclusion that it is a continued-fraction
convergent is the standard Legendre criterion~\cite{khinchin1997}.  The
contribution of the hit-chain theorem is therefore the implication from a
self-prefix configuration to this sharp approximation inequality, not a new
version of Legendre's theorem.

Classical gap and step phenomena for fractional-part and integer-part
sequences are developed, for example, in
\cite{slater1967,fraenkelholzman1995}.  They are a useful comparison for the
candidate sequence in Section~3.  Here the candidates form a
logarithmically perturbed Beatty-type sequence arising from an exact inverse
relation.  Direct control of its ceiling increments yields not only the
eventual two-gap set but also an exact candidate count and a telescoping
formula with a deterministic logarithmic drift.  These conclusions concern
Lambert candidates; they do not assert that either gap endpoint is an actual
prefix hit.

The inversion by the branch $W_{-1}$ is standard~\cite{corlessetal1996}, and
integral and Stieltjes representations of functions built from Lambert $W$
provide related analytic background~\cite{kaluginetal2012}.  We do not
regard the introduction of Lambert $W$ as a contribution by itself.  What is
used here is the problem-specific geometry of the difference of two coupled
$W_{-1}$ inverse roots: its exact layer interpretation, all-orders width
expansion, complete monotonicity, and the resulting arithmetic consequences
for the integer candidates.

Likewise, fast integer arithmetic, algorithmic continued fractions, and
certified ball arithmetic are standard tools
\cite{bachshallit1996,brentzimmermann2010,johansson2017arb}; certified
evaluation of Lambert $W$ is available in arbitrary precision
\cite{johansson2020lambertw}.  The algorithmic result of Section~6 is the
problem-specific locator obtained by combining continued-fraction
separation, outward rational-grid enclosures, Euclidean floor sums, and a
final exact test of the defining prefix inequalities.  Its intermediate
reports form a certified superset, whereas the output after final
verification is the exact set of hits.

The scope of the classification statements also matters.  The complete
classification in this paper is for multiplicatively dependent
\emph{integer} pairs $c,b$.  For a general rational slope
$\log_b c=P/Q$, the exact residue reduction is valid, and the integral
residue branches are classified, but a remaining nonintegral branch may
encode the orbit of a prescribed algebraic irrational under multiplication
by $b$.  No conclusion for those branches is inferred from the integer
classification.

\begin{table}[H]
  \centering
  \small
  \caption{Logical distinctions maintained throughout the paper.}
  \label{tab:logical-status}
  \begin{tabularx}{\textwidth}{@{}
    >{\raggedright\arraybackslash}p{0.18\textwidth}
    >{\raggedright\arraybackslash}X
    >{\raggedright\arraybackslash}X@{}}
    \toprule
    Object & Established here & Not implied \\
    \midrule
    Lambert candidates $m_j$
      & Exact candidate count and eventual two-gap law
      & A candidate need not be an actual prefix hit \\
    \addlinespace
    Actual hits $\Sset_{c,b}$
      & Exact criteria, signed-discrepancy identity, resonance rigidity, and
        a complete classification for dependent integer pairs
      & Infinitude for a fixed multiplicatively independent pair is not
        proved \\
    \addlinespace
    Locator reports $\mathcal R(M,H)$
      & Every genuine hit in the block is reported; final verification gives
        the exact answer
      & A reported index need not be a hit, and the locator does not prove
        that a block contains a hit \\
    \addlinespace
    $\Sset_{2,10}$
      & Exact structural criteria and machine-checkable localization
        certificates
      & Infinitude and $A_{2,10}(X)\sim\log_{10}X$ remain open \\
    \bottomrule
  \end{tabularx}
\end{table}

\subsection{Inputs, scope, and the critical boundary}

The external inputs used in the quantitative arguments have sharply
delimited roles.  Matveev's theorem supplies effective lower
bounds for the relevant linear forms in logarithms.  Tichy--Turnwald
supplies ordinary and logarithmically weighted discrepancy for
$an+\beta\log n$ at finite Diophantine type.  Standard continued-fraction
separation and Legendre's criterion provide the block packing and the
direction of long chains; see, for example,~\cite{khinchin1997}.  The exact
Lambert candidate count, the resonance-shell and endpoint-filling
structures, and their certified-search consequences are then proved for the
self-referential prefix problem in the present paper.

None of these inputs crosses the critical one-sided scale.  In Lambert
coordinates the unresolved condition is
\begin{equation}\label{eq:critical-condition-summary-intro}
  d_+(x_j)<z_j-x_j
  =\frac1{j\log b}
   +O_{c,b}\left(\frac{\log j}{j^2}\right)
  \quad\text{for infinitely many }j.
\end{equation}
Thus the infinitude of $\Sset_{2,10}$ and the logarithmic law
\cref{eq:conjectural-log-law-intro} remain open.  The distinction between
Lambert candidates and actual hits, and between certified localization and
existence, is maintained throughout.

\subsection{Organization}

Section~2 proves the exact prefix and endpoint criteria.  Section~3 develops
Lambert inversion, complete monotonicity, the two-gap law, and deterministic
subcritical discrepancy.  Section~4 gives the exact signed-discrepancy
counting identity.  Section~5 treats the arithmetic dichotomy, fixed-gap
resonance shells, hit chains, nested windows, endpoint filling, and the
coherent skeleton.  Section~6 constructs the certified locator and proves
its interpolated bit-complexity bound.  Section~7 records the remaining
fixed-parameter problems at the critical scale.

%% file: sections/en/02_exact_criterion.tex
\section{Exact prefix and endpoint criteria}

\begin{proposition}[Exact prefix criterion]\label{prop:criterion}
For every $c>1$, $b\ge2$, and $m\ge1$, the equivalence
\cref{eq:shrinking-target-intro} holds.  If $c\ge2$, the condition
$F_{c,b}(m)\ge0$ is automatic.
\end{proposition}

\begin{proof}
The prefix condition is equivalent to the existence of an integer $k\ge0$
such that
\[
  0\le m\log_b c-\log_bm-k
  <\log_b\left(1+\frac1m\right).
\]
If it holds, then $F_{c,b}(m)\ge0$ and the fractional-part inequality follows.
Conversely, under these two conditions one may take
$k=\lfloor F_{c,b}(m)\rfloor$.  If $c\ge2$, then $c^m\ge2^m\ge m$, so
$F_{c,b}(m)\ge0$.
\end{proof}

\begin{proposition}[Endpoint parametrization in the dependent branch]
\label{prop:dependent-endpoint-parametrization}
Suppose $c=d^u$ and $b=d^v$, where $d\ge2$ and $u,v\ge1$.  For
$e\in\{0,1\}$, all endpoint solutions are parametrized exactly by
\begin{equation}\label{eq:dependent-endpoint-parametrization}
\begin{gathered}
  t\in\mathbb Z_{\ge0},\qquad m=d^t-e\ge1,\\
  u(d^t-e)-t\equiv0\pmod v,\qquad
  k=\frac{u(d^t-e)-t}{v}\ge0.
\end{gathered}
\end{equation}
The admissible set of $t$ is eventually periodic.
\end{proposition}

\begin{proof}
The endpoint equation is $d^{um}=(m+e)d^{vk}$.  Every prime divisor of
$m+e$ therefore divides $d$, and comparison of the valuations at every
prime $r\mid d$ gives
\[
  v_r(m+e)=(um-vk)v_r(d).
\]
Thus $t=um-vk$ is a nonnegative integer independent of $r$ and
$m+e=d^t$, proving the displayed parametrization and its converse.  The
sequence $d^t\pmod v$ is eventually periodic, $t\pmod v$ is periodic, and
the inequality defining $k\ge0$ is automatic for all sufficiently large
$t$.
\end{proof}

%% file: sections/en/03_lambert_layers.tex
\section{Lambert inversion, layer widths, and subcritical discrepancy}

Write $\alpha=s+\rho$, where $s=\lfloor\alpha\rfloor$ and
$0<\rho<1$, and put
\[
  C=b^\rho,\qquad L=\log C,
  \qquad
  \Phi(x)=\rho x-\log_bx,
  \qquad
  \Psi(x)=\rho x-\log_b(x+1).
\]
The integer part $sm$ does not affect the prefix event.  This makes it
possible to invert every integer value of $\Phi$ and $\Psi$ separately.
We use the standard notation and real-branch convention for the Lambert
$W$ function from~\cite{corlessetal1996}; in particular, $W_{-1}$ denotes
the large negative real branch on $[-\e^{-1},0)$.

\begin{theorem}[Exact inverse layers]\label{thm:lambert-layer}
Assume $c\ge2$ and $0<\rho<1$.  The indicator of a prefix hit satisfies
\begin{equation}\label{eq:layer-indicator}
  \ind_{\{m\in\Sset_{c,b}\}}
  =\lfloor\Phi(m)\rfloor-\lfloor\Psi(m)\rfloor.
\end{equation}
For all sufficiently large integers $j$, define
\begin{equation}\label{eq:xz-lambert}
  x_j=-\frac1L W_{-1}(-Lb^{-j}),
  \qquad
  z_j=-\frac1L W_{-1}\left(-\frac{L}{Cb^j}\right)-1.
\end{equation}
Then the integers belonging to layer $j$ are exactly
\[
  \mathbb Z\cap[x_j,z_j),
\]
and
\[
  0<z_j-x_j<1.
\]
Thus the unique possible candidate is $m_j=\lceil x_j\rceil$, and it is a
solution if and only if
\begin{equation}\label{eq:exact-lambert-hit}
  \boxed{\lceil x_j\rceil<z_j.}
\end{equation}
Moreover, with $d_+(x)=\lceil x\rceil-x$,
\begin{equation}\label{eq:width-first-order}
  z_j-x_j=\frac1{j\log b}
  +O_{c,b}\left(\frac{\log j}{j^2}\right),
\end{equation}
and $\Sset_{c,b}$ is infinite if and only if
$d_+(x_j)<z_j-x_j$ for infinitely many $j$.
\end{theorem}

\begin{proof}
Let $G(m)=m\alpha-\log_b(m+1)$.  The half-open prefix condition gives
\[
  \ind_{\{m\in\Sset_{c,b}\}}
  =\lfloor F_{c,b}(m)\rfloor-\lfloor G(m)\rfloor.
\]
Since $F_{c,b}(m)=sm+\Phi(m)$ and $G(m)=sm+\Psi(m)$,
\cref{eq:layer-indicator} follows.  Also
$0<\Phi(m)-\Psi(m)=\delta_b(m)\le1$, so a hit is equivalent to the
existence of a unique integer $j$ with
\[
  \Psi(m)<j\le\Phi(m).
\]

The equation $\Phi(x)=j$ is
$(-Lx)\e^{-Lx}=-Lb^{-j}$.  The large positive root uses the real branch
$W_{-1}$ and is $x_j$ in \cref{eq:xz-lambert}.  Similarly, after setting
$u=z+1$, the equation $\Psi(z)=j$ gives the displayed formula for $z_j$.
For all sufficiently large $j$, both roots are defined and lie in a range on
which $\Phi$ and $\Psi$ are strictly increasing.  Hence
$j\le\Phi(m)$ is equivalent to $m\ge x_j$, while $\Psi(m)<j$ is equivalent
to $m<z_j$.

Since $\Psi(x_j)<\Phi(x_j)=j$, one has $z_j>x_j$.  For sufficiently large
$x_j$,
\[
  \Psi(x_j+1)-\Phi(x_j)
  =\log_b\frac{Cx_j}{x_j+2}>0,
\]
because $(C-1)x_j>2$.  Thus $z_j<x_j+1$, proving uniqueness and
\cref{eq:exact-lambert-hit}.

Finally, the defining equations give
\[
  Lx_j-\log x_j=j\log b,
  \qquad
  Lz_j-\log(z_j+1)=j\log b.
\]
They first imply
$x_j,z_j=(j\log b)/L+O_{c,b}(\log j)$.  The mean-value theorem applied to
$\Phi$ gives, for some $\xi_j\in(x_j,z_j)$,
\[
  z_j-x_j
  =\frac{\log(1+1/z_j)}{L-1/\xi_j},
\]
which yields \cref{eq:width-first-order}.  Since the omitted initial layers
are finite, the last equivalence follows from
\cref{eq:exact-lambert-hit}.
\end{proof}

\begin{proposition}[All-orders layer width]\label{prop:all-orders-width}
Let $B=\log b$, $J=jB$, $Y_j=Lx_j$, and $w_j=z_j-x_j$.  Whenever the two
large real branches exist,
\begin{equation}\label{eq:exact-width-W}
  w_j=\frac{-W_{-1}(-L\e^{-J-L})+W_{-1}(-L\e^{-J})-L}{L}.
\end{equation}
For all sufficiently large $j$, the following series is convergent:
\begin{equation}\label{eq:lagrange-width}
  \boxed{
  w_j=\sum_{n\ge1}\frac1{nL}[z^{n-1}]
  \left(\frac{z(L+z)}{\e^z-1}\right)^nY_j^{-n}.}
\end{equation}
In particular,
\begin{align}\label{eq:width-y-expansion}
  w_j={}&Y_j^{-1}
  +\left(1-\frac L2\right)Y_j^{-2}
  +\left(1-\frac{3L}{2}+\frac{L^2}{3}\right)Y_j^{-3}\notag\\
  &+\left(1-3L+\frac{11L^2}{6}-\frac{L^3}{4}\right)Y_j^{-4}
  +O_L(Y_j^{-5}).
\end{align}
If $\tau_j=\log(J/L)$ and
\[
  Q_L(\tau)=\tau^2-(3-L)\tau+1-\frac{3L}{2}+\frac{L^2}{3},
\]
then
\begin{equation}\label{eq:width-j-expansion}
  w_j=\frac1{jB}
  -\frac{\tau_j-1+L/2}{j^2B^2}
  +\frac{Q_L(\tau_j)}{j^3B^3}
  +O_{c,b}\left(\frac{(\log j)^3}{j^4}\right).
\end{equation}
Every further finite order is obtained effectively from
\cref{eq:lagrange-width}.
\end{proposition}

\begin{proof}
Set $y_j=z_j+1$.  Then
\[
  Lx_j-\log x_j=J,
  \qquad
  Ly_j-\log y_j=J+L,
\]
which gives \cref{eq:exact-width-W}.  Put
$v_j=Lw_j$ and $t_j=Y_j^{-1}$.  Subtracting the two equations yields the
exact implicit relation
\begin{equation}\label{eq:v-implicit}
  v_j=\log\bigl(1+t_j(L+v_j)\bigr).
\end{equation}
The analytic implicit-function theorem gives an analytic branch at
$(v,t)=(0,0)$.  This is the actual positive width branch: for small $t>0$,
$H_t(v)=\e^v-1-t(L+v)$ is strictly increasing on $v\ge0$, is negative at
zero, and tends to infinity.

More explicitly, the implicit-function theorem supplies a number
$t_0=t_0(L)>0$ for which this branch, and hence its Taylor series, is
convergent throughout $|t|<t_0$.  Since $t_j=Y_j^{-1}\to0$, every
sufficiently large layer lies in that common disk of convergence.

Exponentiating \cref{eq:v-implicit} gives
\[
  v=t\,\frac{v(L+v)}{\e^v-1}.
\]
Lagrange--B\"urmann inversion proves \cref{eq:lagrange-width}; after
truncation at order $R$, analyticity supplies a remainder
$O_{L,R}(Y_j^{-R-1})$.  Expanding
the first four coefficients gives \cref{eq:width-y-expansion}.  Finally,
$Y_j=J+\log x_j$ and
\[
  \log x_j=\tau_j+\frac{\tau_j}{J}
  +O_L\left(\frac{\tau_j^2}{J^2}\right).
\]
Substitution into \cref{eq:width-y-expansion} yields
\cref{eq:width-j-expansion}.
\end{proof}

\subsection{Complete monotonicity and the two-gap law for candidates}

The all-orders expansion above has a global counterpart: the layer width
has no oscillatory derivatives.  This also forces a rigid two-gap structure
on the integer candidates, although it does not decide which candidates are
actual prefix hits.  Complete monotonicity and integral representations for
functions related to Lambert $W$ have an established theory
\cite{kaluginetal2012}; the proposition below is the tailored statement for
the difference of the two large inverse roots defining one prefix layer.

\begin{proposition}[Complete monotonicity of the layer width]
\label{prop:width-complete-monotonicity}
Let $B=\log b$, and let $G$ be the large-branch inverse of
\[
  u\longmapsto Lu-\log u
\]
on $(1/L,\infty)$.  For $J>1+\log L$, put
\[
  R(J)=\frac1{LG(J)-1},
  \qquad
  W(J)=G(J+L)-G(J)-1.
\]
Then
\begin{equation}\label{eq:width-integral-representation}
  \boxed{W(J)=\frac1L\int_J^{J+L}R(u)\,du},
  \qquad
  R'(J)=-R(J)^2(1+R(J)).
\end{equation}
In particular,
\begin{equation}\label{eq:width-complete-monotonicity}
  \boxed{(-1)^nW^{(n)}(J)>0\qquad(n\ge0).}
\end{equation}
For every layer for which $x_j=G(jB)$ is on this branch, $w_j=z_j-x_j$
equals $W(jB)$ and satisfies the strict nonasymptotic bounds
\begin{equation}\label{eq:width-nonasymptotic-bounds}
  \boxed{
  \frac1{L(z_j+1)-1}<w_j<\frac1{Lx_j-1}.}
\end{equation}
Moreover,
\begin{equation}\label{eq:width-one-exact-threshold}
  \boxed{w_j<1\quad\Longleftrightarrow\quad
  x_j>\frac2{\e^L-1}=\frac2{C-1}.}
\end{equation}
If $\Delta a_j=a_{j+1}-a_j$, then throughout the discrete large-branch
range
\begin{equation}\label{eq:width-discrete-complete-monotonicity}
  (-1)^n\Delta^nw_j>0\qquad(n\ge0),
\end{equation}
and, for each fixed $n$,
\begin{equation}\label{eq:width-difference-asymptotic}
  (-1)^n\Delta^nw_j
  =\frac{n!}{B j^{n+1}}
  +O_{c,b,n}\left(\frac{\log j}{j^{n+2}}\right).
\end{equation}
Thus $(w_j)$ is strictly decreasing and strictly convex.
\end{proposition}

\begin{proof}
Implicit differentiation of $LG(J)-\log G(J)=J$ gives
\[
  G'(J)=\frac{G(J)}{LG(J)-1}=\frac1L(1+R(J)).
\]
Integrating $G'$ over $[J,J+L]$ proves the first identity in
\cref{eq:width-integral-representation}; differentiating $R$ proves the
second one.  Define recursively
\[
  P_0(t)=t,
  \qquad
  P_{n+1}(t)=t^2(1+t)P_n'(t).
\]
Induction gives $(-1)^nR^{(n)}(J)=P_n(R(J))$.  Every $P_n$ is a nonzero
polynomial with nonnegative coefficients, so every one of these quantities
is positive.  Differentiating the integral representation proves
\cref{eq:width-complete-monotonicity}.

Since $R$ is strictly decreasing, its integral average lies strictly
between its endpoint values; this is
\cref{eq:width-nonasymptotic-bounds}.  Next, $w_j<1$ is equivalent to
$G(jB+L)<x_j+2$.  The large branch is increasing, and
\[
  \bigl(L(x_j+2)-\log(x_j+2)\bigr)
  -\bigl(Lx_j-\log x_j\bigr)
  =2L-\log\left(1+\frac2{x_j}\right).
\]
Comparison with $L$ proves \cref{eq:width-one-exact-threshold}, including
the equality case.

Repeated use of the fundamental theorem of calculus gives
\[
  \Delta^nf(j)=
  \int_{[0,1]^n}f^{(n)}(j+t_1+\cdots+t_n)
  \,dt_1\cdots dt_n.
\]
Applied to $f(j)=W(jB)$, after the harmless positive scaling by $B$, this
proves \cref{eq:width-discrete-complete-monotonicity}.  Finally,
$R(J)=J^{-1}+O_L((\log J)/J^2)$, while the recursion gives
$P_n(t)=n!t^{n+1}+O_n(t^{n+2})$.  Substitution in the integral formula
proves \cref{eq:width-difference-asymptotic}.
\end{proof}

\begin{lemma}[Ceiling increments]
\label{lem:ceiling-increments}
Let $x<y$ be real numbers and let $k\in\mathbb Z$.  If
\begin{equation}\label{eq:ceiling-increment-hypothesis}
  k<y-x<k+1,
\end{equation}
then
\begin{equation}\label{eq:ceiling-increment-conclusion}
  \lceil y\rceil-\lceil x\rceil\in\{k,k+1\}.
\end{equation}
This includes without exception the cases in which $x$ or $y$ is an
integer.
\end{lemma}

\begin{proof}
The strict inequalities in
\cref{eq:ceiling-increment-hypothesis} give
\[
  x+k<y<x+k+1.
\]
Monotonicity of the ceiling function, together with
$\lceil x+r\rceil=\lceil x\rceil+r$ for $r\in\mathbb Z$, therefore yields
\[
  \lceil x\rceil+k
  =\lceil x+k\rceil
  \le\lceil y\rceil
  \le\lceil x+k+1\rceil
  =\lceil x\rceil+k+1.
\]
The middle difference is an integer, which proves the claim.  The argument
does not assume that either endpoint is nonintegral.
\end{proof}

The next result is reminiscent of classical gap laws for rotations and
integer-part sequences~\cite{slater1967,fraenkelholzman1995}.  Its input,
however, is not a pure rotation or an exact Beatty sequence: Lambert
inversion introduces a deterministic logarithmic perturbation, which is
responsible for the drift in \cref{eq:large-candidate-gap-count} below.

\begin{theorem}[Two-gap law and exact candidate count]
\label{thm:lambert-candidate-two-gap}
Let
\[
  a=\frac BL=\frac1\rho,
  \qquad k=\lfloor a\rfloor,
  \qquad \vartheta=\{a\},
\]
and set $m_j=\lceil x_j\rceil$.  The real increments
$d_j=x_{j+1}-x_j$ satisfy
\begin{equation}\label{eq:candidate-real-gap}
  d_j=a+\frac1L\int_{jB}^{jB+B}R(u)\,du,
\end{equation}
  so they are strictly decreasing and tend to $a$ from above.  If
  $h=k+1$ and $\delta=h-a$, then the exact equivalence
\begin{equation}\label{eq:candidate-two-gap-threshold}
  \boxed{d_j<h\quad\Longleftrightarrow\quad
  x_j>\frac{h}{\e^{L\delta}-1}}
\end{equation}
  holds.  Consequently, from some layer $j_0$ onward,
\begin{equation}\label{eq:candidate-two-gap}
  \boxed{m_{j+1}-m_j\in\{k,k+1\}.}
\end{equation}

For $N>j_0$, let
\[
  G_+(N)=\#\{j_0\le j<N:m_{j+1}-m_j=k+1\}.
\]
Then
\begin{align}
  G_+(N)
  &=m_N-m_{j_0}-k(N-j_0),
    \label{eq:large-candidate-gap-exact}\\
  &=\vartheta N+\frac{\log N}{L}+O_{c,b}(1).
    \label{eq:large-candidate-gap-count}
\end{align}
More uniformly, for all sufficiently large $s>r$,
\begin{equation}\label{eq:large-candidate-gap-interval}
  \#\{r\le j<s:m_{j+1}-m_j=k+1\}
  =\vartheta(s-r)+\frac1L\log\frac sr
  +d_+(x_s)-d_+(x_r)
  +O_{c,b}\left(\frac{\log r}{r}\right).
\end{equation}
In particular, after subtracting its linear term and logarithmic drift, the
interval count has absolute discrepancy less than
$1+O_{c,b}((\log r)/r)$, uniformly in $s>r$.  Thus the gap word is
asymptotically one-balanced after the deterministic logarithmic correction.
In particular, even when $a$ is an integer, the longer gap occurs
infinitely often, but only $L^{-1}\log N+O(1)$ times up to layer $N$.

Finally, if $j_0$ is chosen so that the large branch is increasing and the
one-candidate conclusion holds for every $j\ge j_0$, then for every integer
$M>1/L$,
\begin{equation}\label{eq:exact-lambert-candidate-count}
  \boxed{
  \#\{j\ge j_0:m_j\le M\}
  =\max\left\{0,
  \left\lfloor\rho M-\log_bM\right\rfloor-j_0+1\right\}.}
\end{equation}
Consequently, for every sufficiently large integer $M$, the candidate set
has the exact perturbed-Beatty indicator
\begin{equation}\label{eq:lambert-candidate-indicator}
  \boxed{
  \ind_{\{M=m_j\text{ for some }j\}}
  =\lfloor\Phi(M)\rfloor-\lfloor\Phi(M-1)\rfloor.}
\end{equation}
\end{theorem}

\begin{proof}
Integrating the formula for $G'$ over $[jB,(j+1)B]$ proves
\cref{eq:candidate-real-gap}.  Complete monotonicity of $R$ shows that
$d_j\downarrow a$.  Since the large branch is increasing, $d_j<h$ is
equivalent to
\[
  L(x_j+h)-\log(x_j+h)>(j+1)B.
\]
After subtracting $Lx_j-\log x_j=jB$, this becomes
\[
  Lh-\log\left(1+\frac h{x_j}\right)>B,
\]
  which is exactly \cref{eq:candidate-two-gap-threshold}.  Since
  $d_j>a\ge k$, the threshold eventually gives
  $k<d_j<k+1$.  Applying \cref{lem:ceiling-increments} to
  $(x_j,x_{j+1})$ proves \cref{eq:candidate-two-gap}.  This also covers
  $a\in\mathbb Z$, integral values of $x_j$, and every strict-threshold
  boundary case.  If $b$ and $c$ are supplied by effective real data that
  decide the displayed inequalities, the explicit threshold in
  \cref{eq:candidate-two-gap-threshold} also provides an effective search
  for a valid $j_0$; no such effectiveness assertion is needed for an
  arbitrary abstract real parameter $c$.

On this range, the indicator of the longer gap is
$m_{j+1}-m_j-k$, so summation telescopes and gives
\cref{eq:large-candidate-gap-exact}.  Inversion of
$Lx_j-\log x_j=jB$ gives
\[
  x_j=aj+\frac{\log j}{L}+O_{c,b}(1),
  \qquad m_j=x_j+O(1).
\]
This proves \cref{eq:large-candidate-gap-count}; subtracting the analogous
formula at two endpoints, and retaining
$m_j-x_j=d_+(x_j)$, proves \cref{eq:large-candidate-gap-interval}.  The
remainder is uniform because $(\log s)/s\ll(\log r)/r$ for sufficiently
large $s>r$.

For the final assertion, monotonicity of
$\Phi(x)=\rho x-\log_bx$ on the large branch gives, for integer $M$,
\[
  m_j\le M
  \quad\Longleftrightarrow\quad
  x_j\le M
  \quad\Longleftrightarrow\quad
  j\le\rho M-\log_bM.
\]
Counting the eligible integers $j\ge j_0$ proves
\cref{eq:exact-lambert-candidate-count}.  For a real upper bound one must
first replace it by its integer part.  Taking the difference of the exact
counts at $M$ and $M-1$ proves \cref{eq:lambert-candidate-indicator}; the
difference is $0$ or $1$ because $\Phi$ is increasing with increment less
than $1$ in this range.
\end{proof}

\begin{corollary}[The $3$--$4$ candidate law for $(2,10)$]
\label{cor:two-candidate-two-gap}
For $(c,b)=(2,10)$ one has
\[
  m_1=6,
  \qquad m_2=10,
  \qquad m_{j+1}-m_j\in\{3,4\}\quad(j\ge1).
\]
If
\[
  G_4(N)=\#\{1\le j<N:m_{j+1}-m_j=4\},
\]
then
\begin{equation}\label{eq:two-candidate-four-count}
  \boxed{G_4(N)=m_N-3N-3
  =(\log_2 10-3)N+\log_2N+O(1).}
\end{equation}
For every integer $M\ge2$ the total number of Lambert candidates not
exceeding $M$ is exactly
\begin{equation}\label{eq:two-exact-candidate-count}
  \boxed{
  \#\{j\ge1:m_j\le M\}
  =\left\lfloor M\log_{10}2-\log_{10}M\right\rfloor.}
\end{equation}
These statements count possible candidates, not actual members of
$\Sset_{2,10}$.
\end{corollary}

\begin{proof}
The exact comparisons
\[
  \frac{2^5}{5}<10<\frac{2^6}{6},
  \qquad
  \frac{2^9}{9}<100<\frac{2^{10}}{10}
\]
give $5<x_1<6$ and $9<x_2<10$.  Thus $m_1=6$ and $m_2=10$.
Here $a=\log_2 10\in(3,4)$, and the threshold in
\cref{eq:candidate-two-gap-threshold} for $h=4$ equals
\[
  \frac4{\exp((\log2)(4-\log_2 10))-1}
  =\frac4{8/5-1}=\frac{20}{3}.
\]
Since $x_2>9$, all later real gaps lie in $(\log_2 10,4)$; together with
$m_2-m_1=4$, this proves the two-gap assertion.  Telescoping gives the
first identity in \cref{eq:two-candidate-four-count}, and
\cref{eq:large-candidate-gap-count} gives its asymptotic form.
Finally, the large branch and the one-candidate property already hold from
$j=1$; \cref{eq:exact-lambert-candidate-count} with $j_0=1$ proves
\cref{eq:two-exact-candidate-count}.
\end{proof}

\begin{corollary}[A second-order one-sided critical criterion]
\label{cor:second-order-critical-criterion}
Retain the notation of \cref{thm:lambert-layer,prop:all-orders-width}, and
define
\[
  \Theta_j
  =j^2\left(d_+(x_j)-\frac1{jB}\right)
  +\frac{\log j+\log(1/\rho)-1+L/2}{B^2}.
\]
Then
\[
  \liminf_{j\to\infty}\Theta_j<0
  \quad\Longrightarrow\quad
  \Sset_{c,b}\text{ is infinite},
\]
whereas
\[
  \liminf_{j\to\infty}\Theta_j>0
  \quad\Longrightarrow\quad
  \Sset_{c,b}\text{ is finite}.
\]
The equality case is not decided at second order.  Successive explicit
terms in \cref{eq:width-j-expansion} give higher-order sufficient criteria
whenever the first nonvanishing normalized remainder is separated from
zero; they do not by themselves decide a fully degenerate boundary case.
\end{corollary}

\begin{proof}
Since $\tau_j=\log(jB/L)=\log j+\log(1/\rho)$,
\cref{eq:width-j-expansion} gives
\[
  j^2\bigl(d_+(x_j)-(z_j-x_j)\bigr)
  =\Theta_j+O_{c,b}\left(\frac{(\log j)^2}{j}\right).
\]
The error tends to zero, while the exact layer criterion is
$d_+(x_j)<z_j-x_j$.  A negative lower limit therefore supplies infinitely
many hits separated from the boundary by a fixed normalized margin; a
positive lower limit excludes every sufficiently large layer.
\end{proof}

\begin{corollary}[The exact reduction for $2^m$]\label{cor:two-lambert}
For $(c,b)=(2,10)$,
\[
  x_j=-\frac{W_{-1}(-\log2\,10^{-j})}{\log2},
  \qquad
  z_j=-1-\frac{W_{-1}(-\log2/(2\cdot10^j))}{\log2}.
\]
Hence $\Sset_{2,10}$ is infinite if and only if
\begin{equation}\label{eq:two-exact-one-sided}
  \lceil x_j\rceil<z_j
\end{equation}
for infinitely many $j$.  If
\[
  \lambda_{2,10}=\liminf_{j\to\infty}j(\lceil x_j\rceil-x_j),
\]
then
\[
  \lambda_{2,10}<\frac1{\log10}\Longrightarrow\Sset_{2,10}\text{ infinite},
  \qquad
  \lambda_{2,10}>\frac1{\log10}\Longrightarrow\Sset_{2,10}\text{ finite}.
\]
At equality, successive explicit terms in
\cref{eq:width-j-expansion} yield higher-order sufficient tests when the
corresponding normalized remainder is separated from zero.  No such sign
determination for this fixed sequence is currently known.
\end{corollary}

\begin{proof}
Here $\rho=\log_{10}2$, so $C=2$ and $L=\log2$.  Substitute these values in
\cref{thm:lambert-layer}.  Since
$j(z_j-x_j)\to1/\log10$, the two strict implications follow directly from
the definition of the lower limit.
\end{proof}

For a real sequence $u=(u_j)$, write
\begin{align*}
  D_N^{\mathrm{arith},*}(u)
  &=\sup_{0\le v\le1}\left|
    \frac1N\sum_{j\le N}\ind_{[0,v)}(\{u_j\})-v\right|,\\
  D_N^{\log,*}(u)
  &=\sup_{0\le v\le1}\left|
    \frac1{H_N}\sum_{j\le N}
    \frac{\ind_{[0,v)}(\{u_j\})}{j}-v\right|,
  \qquad H_N=\sum_{j\le N}\frac1j.
\end{align*}

\begin{lemma}[Tichy--Turnwald finite-type estimate]
\label{lem:tichy-turnwald-finite-type}
Let $a\notin\mathbb Q$, let $\beta>0$, and suppose that some $\eta\ge1$
has the following property: for every $\epsilon>0$ there is
$c(a,\epsilon)>0$ such that
\[
  \|qa\|\ge c(a,\epsilon)q^{-\eta-\epsilon}
  \qquad(q\ge1).
\]
Then, for every $0<\sigma<1/(\eta+1)$, every fixed
$\gamma\in\mathbb R$, and every half-open circular interval $I\subset\T$,
uniformly in $I$,
\begin{align}
 \#\{n\le N:\{an+\beta\log n+\gamma\}\in I\}
 &=N|I|+O_{a,\beta,\sigma}(N^{1-\sigma}),
 \label{eq:tt-ordinary}\\
 \sum_{n\le N}\frac{\ind_I(\{an+\beta\log n+\gamma\})}{n}
 &=|I|H_N+O_{a,\beta,\sigma}(1).
 \label{eq:tt-logarithmic}
\end{align}
\end{lemma}

\begin{proof}
Remark~2 in Section~2 of Tichy--Turnwald
\cite[pp.~356--357]{tichyturnwald1986} gives
\[
 D_N^{\mathrm{arith},*}(an+\beta\log n)
 \ll_{a,\beta,\epsilon}
 N^{-1/(\eta+1)+\epsilon}.
\]
Choosing $\epsilon<1/(\eta+1)-\sigma$ proves the initial-interval form of
\cref{eq:tt-ordinary}.  A general circular interval is the difference of
two initial intervals (or the complement of such a difference), so the
same estimate holds uniformly, with at most an absolute factor two.  A
fixed $\gamma$ merely rotates the test interval.

Tichy--Turnwald normalize logarithmic discrepancy by
$H_N=\sum_{n\le N}1/n$ and prove, using their formula~(1) on p.~356,
\[
  D_N^{\log,*}(an+\beta\log n)\ll_{a,\beta}\frac1{\log N}.
\]
Multiplication by $H_N\asymp\log N$, followed by the same two-endpoint
reduction, gives \cref{eq:tt-logarithmic}.  This records explicitly the
finite-type convention, the sign assumption on $\beta$, the interval
uniformity, and the normalization used from the cited paper.
\end{proof}

\begin{theorem}[Power-saving discrepancy for algebraic Lambert sequences]
\label{prop:two-log-discrepancy}
Let $b\ge2$ be an integer and let $c>1$ be a fixed positive real algebraic
number such that $\alpha=\log_b c$ is irrational.  Write
\[
  s=\lfloor\alpha\rfloor,\qquad \rho=\alpha-s,\qquad
  B=\log b,\qquad L=\rho B=\log(c/b^s),
\]
and, for all sufficiently large $j$, let $x_j$ be the large real root of
\begin{equation}\label{eq:general-lambert-root}
  Lx_j-\log x_j=jB.
\end{equation}
There is an effectively computable $\eta=\eta(c,b)\ge1$ such that, for every
\[
  0<\sigma<\frac1{\eta+1},
\]
uniformly over all circular intervals $I\subset\T$,
\begin{align}
  \#\{j\le N:\{x_j\}\in I\}
  &=N|I|+O_{c,b,\sigma}(N^{1-\sigma}),
  \label{eq:ordinary-interval-discrepancy}\\
  \sum_{j\le N}\frac{\ind_I(\{x_j\})}{j}
  &=|I|H_N+O_{c,b,\sigma}(1),
  \qquad H_N=\sum_{j\le N}\frac1j.
  \label{eq:log-interval-discrepancy}
\end{align}
Equivalently, the ordinary and logarithmically weighted star discrepancies
satisfy
\begin{equation}\label{eq:log-star-discrepancy}
  D_N^{\mathrm{arith},*}(x)\ll_{c,b,\sigma} N^{-\sigma},
  \qquad
  D_N^{\log,*}(x)\ll_{c,b,\sigma}\frac1{\log N}.
\end{equation}
The assertion is deterministic for every prescribed algebraic pair under
the stated irrationality hypothesis; it contains no almost-everywhere
qualification.  In particular, it applies to $(c,b)=(2,10)$.
\end{theorem}

\begin{proof}
Set
\[
  a=\frac BL=\frac1\rho,\qquad \beta=\frac1L.
\]
We first verify the Diophantine hypothesis for this fixed $a$.  For
$q\ge1$, let $p$ be a nearest integer to $qa$ and put
\[
  \Lambda=qB-pL=(q+sp)\log b-p\log c.
\]
The number $\Lambda$ is nonzero: equality would give $a=p/q$ and hence
$\alpha=s+1/a\in\mathbb Q$.  Matveev's lower bound for a nonzero linear
form in logarithms of fixed algebraic numbers
\cite{matveev2000} supplies effective constants $C_0\ge1$ and $C_1>0$ for
which
\[
  |\Lambda|\ge
  C_1\max\{3,|q+sp|,|p|\}^{-C_0}.
\]
Here $|p|\le aq+1/2$, so, after changing the constant,
\begin{equation}\label{eq:finite-type-lambert-general}
  \|qa\|
  =\frac{|qB-pL|}{L}
  \ge C_2q^{-C_0}\qquad(q\ge1).
\end{equation}
Thus $a$ has finite approximation type; we may take $\eta=C_0$.

Now define
\[
  y_j=aj+\beta\log j+\gamma,
  \qquad \gamma=\frac{\log a}{L}.
\]
Since $\beta=1/L>0$, \cref{eq:finite-type-lambert-general} satisfies the
hypotheses of \cref{lem:tichy-turnwald-finite-type}.  Hence, for every
$0<\sigma<1/(\eta+1)$ and every circular interval $I$, uniformly in $I$,
\begin{align}
  \#\{j\le N:\{y_j\}\in I\}
  &=N|I|+O_{c,b,\sigma}(N^{1-\sigma}),
  \label{eq:y-ordinary-discrepancy}\\
  \sum_{j\le N}\frac{\ind_I(\{y_j\})}{j}
  &=|I|H_N+O_{c,b,\sigma}(1).
  \label{eq:y-log-discrepancy}
\end{align}
Adding the fixed constant $\gamma$ only rotates the test interval and hence
does not affect these uniform estimates.

It remains to transfer the estimates from $y_j$ to the exact Lambert root.
Extend $x_j$ to the large real root $x(t)$ of
\[
  Lx(t)-\log x(t)=Bt.
\]
The defining equation first gives $x(t)=at+O_{c,b}(\log t)$, and then
\begin{equation}\label{eq:x-slow-perturbation}
  x(t)=at+\beta\log t+\gamma+r(t),
  \qquad
  r(t)=\frac1L\log\frac{x(t)}{at}
      =O_{c,b}\left(\frac{\log t}{t}\right).
\end{equation}
Implicit differentiation also gives
\[
  r'(t)
  =\frac{a}{Lx(t)-1}-\frac1{Lt}
  =-\frac{a(\log x(t)-1)}
  {Bt\,[Bt+\log x(t)-1]}
  =O_{c,b}\left(\frac{\log t}{t^2}\right).
\]
Only the first estimate in \cref{eq:x-slow-perturbation} is needed below,
but the derivative estimate records that the perturbation is genuinely
slow.

Fix a dyadic block $B_M=\{j:M\le j<2M\}$, where $M$ is a power of two.
For a suitable fixed $R>0$ and all sufficiently large $M$,
\[
  |x_j-y_j|\le\delta_M,
  \qquad
  \delta_M=R\frac{\log(2M)}M
  \qquad(j\in B_M).
\]
If $\ind_I(\{x_j\})\ne\ind_I(\{y_j\})$, the circular segment joining
$\{y_j\}$ to $\{x_j\}$ crosses an endpoint of $I$.  Hence $\{y_j\}$ lies
in the $\delta_M$-neighborhood of $\partial I$, a union of at most four
ordinary intervals having total length at most $4\delta_M$.  Applying
\cref{eq:y-ordinary-discrepancy} at $M$ and $2M$ to these intervals gives
\begin{equation}\label{eq:boundary-crossing-count}
  \#\{j\in B_M:
  \ind_I(\{x_j\})\ne\ind_I(\{y_j\})\}
  \ll_{c,b,\sigma} \log M+M^{1-\sigma},
\end{equation}
uniformly in $I$.  Summing \cref{eq:boundary-crossing-count} over dyadic
blocks up to $N$ gives $O_{c,b,\sigma}(N^{1-\sigma})$, which, together with
\cref{eq:y-ordinary-discrepancy}, proves
\cref{eq:ordinary-interval-discrepancy}.

On $B_M$ the total logarithmic weight of all crossing indices is at most
\[
  \frac1M O_{c,b,\sigma}(\log M+M^{1-\sigma})
  =O_{c,b,\sigma}\left(\frac{\log M}{M}+M^{-\sigma}\right).
\]
Both expressions are summable over dyadic $M$.  Consequently the total
weighted difference between the $x_j$ and $y_j$ indicators is bounded,
uniformly in $N$ and $I$.  Combining this with
\cref{eq:y-log-discrepancy} proves
\cref{eq:log-interval-discrepancy}; the star-discrepancy formulation follows
by taking initial intervals and using $H_N\asymp\log N$.
\end{proof}

\begin{corollary}[A deterministic slow shrinking target]
\label{cor:algebraic-slow-target}
Under the hypotheses of \cref{prop:two-log-discrepancy}, let
$\psi:[1,\infty)\to(0,1]$ be eventually nonincreasing, with
$\psi(t)\to0$ and
\[
  \psi(N)\log N\longrightarrow\infty.
\]
Then $d_+(x_j)<\psi(j)$ for infinitely many $j$.  More precisely,
\begin{equation}\label{eq:slow-target-lower-bound}
  \sum_{j\le N}\frac{\ind_{\{d_+(x_j)<\psi(j)\}}}{j}
  \ge \psi(N)H_N-O_{c,b,\sigma}(1).
\end{equation}
This is an infinitude theorem for every prescribed algebraic pair covered
by \cref{prop:two-log-discrepancy}, but its windows remain much larger than
the true prefix width $\asymp1/j$.
\end{corollary}

\begin{proof}
Uniformity in \cref{eq:log-interval-discrepancy} also controls endpoint
atoms.  Indeed, enclosing any fixed $t\in\T$ in an interval of length
$\varepsilon$ and then letting $\varepsilon\downarrow0$ gives
\[
  \sum_{\substack{j\le N\\\{x_j\}=t}}\frac1j
  =O_{c,b,\sigma}(1),
\]
uniformly in $t$.  For $0<w<1$, the event $d_+(x_j)<w$ is the union of
$\{\{x_j\}=0\}$ and $\{\{x_j\}\in(1-w,1)\}$; hence changing endpoint
conventions costs only the displayed $O(1)$ atom bound.  Since
$\psi(j)\ge\psi(N)$ for $j\le N$, applying
\cref{eq:log-interval-discrepancy} to the circular interval of length
$\psi(N)$ proves \cref{eq:slow-target-lower-bound}.  Its right-hand side
tends to infinity, which is impossible if only finitely many indices
contribute.
\end{proof}

\begin{theorem}[Subcritical moving targets for algebraic bases]
\label{thm:two-subcritical-moving}
Let $\eta$ be as in \cref{prop:two-log-discrepancy}, fix
$0<\sigma<1/(\eta+1)$, and define
\[
  d_+(x)=\lceil x\rceil-x,
  \qquad
  C_{\kappa,\tau}(N)
  =\#\{j\le N:d_+(x_j)<\kappa j^{-\tau}\}.
\]
If $0<\tau<\sigma$ and $\kappa>0$, then
\begin{equation}\label{eq:two-subcritical-asymptotic}
  C_{\kappa,\tau}(N)
  =\frac{\kappa}{1-\tau}N^{1-\tau}
  +O_{c,b,\kappa,\tau,\sigma}
   \left(N^{1-(\sigma+\tau)/2}\right).
\end{equation}
The same formula holds for $\tau=0$ and $0<\kappa\le1$; in that case the
error may be sharpened to $O_{c,b,\sigma}(N^{1-\sigma})$.
In particular, every sufficiently large dyadic block contains
$\asymp_{c,b,\kappa,\tau}N^{1-\tau}$ such subcritical hits.
\end{theorem}

\begin{proof}
Put $u_j=d_+(x_j)\in[0,1)$.  We do not assume that the general algebraic
Lambert root is never an integer.  Uniformity in
\cref{eq:ordinary-interval-discrepancy}, applied to an interval of length
$\varepsilon$ containing any prescribed $t\in\T$ and followed by
$\varepsilon\downarrow0$, gives the endpoint-atom bound
\[
  \#\{j\le N:\{x_j\}=t\}=O_{c,b,\sigma}(N^{1-\sigma})
\]
uniformly in $t$.  For $0<w<1$,
\[
  \{u_j<w\}
  =\{\{x_j\}=0\}\cup\{\{x_j\}\in(1-w,1)\}.
\]
The cases $w=0,1$ are immediate.  Reflecting the second interval and using
the atom bound therefore shows that, uniformly for $0\le w\le1$ and
$1\le H\le M$,
\begin{equation}\label{eq:fixed-window-short-block}
  \#\{M\le j<M+H:u_j<w\}
  =Hw+O_{c,b,\sigma}(M^{1-\sigma}).
\end{equation}
Indeed, this is obtained by subtracting the two corresponding prefix
counts at $M+H-1$ and $M-1$.

Assume first $0<\tau<\sigma$; finitely many indices for which
$\kappa j^{-\tau}>1$ may be discarded.  Put
\[
  \lambda=\frac{\sigma-\tau}{2}>0.
\]
Starting from a sufficiently large integer $M_0$, define consecutive blocks
by
\[
  H_r=\lfloor M_r^{1-\lambda}\rfloor,
  \qquad M_{r+1}=M_r+H_r.
\]
On the block $M=M_r\le j<M+H_r$, monotonicity gives
\[
  \kappa(M+H_r)^{-\tau}
  \le\kappa j^{-\tau}\le\kappa M^{-\tau}.
\]
The difference between the two fixed thresholds is, by the mean-value
theorem,
\[
  \kappa M^{-\tau}-\kappa(M+H_r)^{-\tau}
  =O_{\kappa,\tau}(H_rM^{-\tau-1}).
\]
Applying \cref{eq:fixed-window-short-block} to both thresholds therefore
yields the moving-window block formula
\begin{equation}\label{eq:moving-window-adaptive-block}
\begin{aligned}
  &\#\{M\le j<M+H_r:u_j<\kappa j^{-\tau}\}\\
  &\qquad=\kappa H_rM^{-\tau}
  +O_{\kappa,\tau}(H_r^2M^{-\tau-1})
  +O_{c,b,\sigma}(M^{1-\sigma}).
\end{aligned}
\end{equation}

Since $H_r\asymp M_r^{1-\lambda}$, summation over the consecutive block
beginnings is bounded by the corresponding integrals with density
$dt/t^{1-\lambda}$.  Hence
\begin{align}
  \sum_{M_r\le N}M_r^{1-\sigma}
  &\ll N^{1+\lambda-\sigma},
  \label{eq:sum-adaptive-disc}\\
  \sum_{M_r\le N}H_r^2M_r^{-\tau-1}
  &\ll N^{1-\lambda-\tau}.
  \label{eq:sum-adaptive-motion}
\end{align}
The same mean-value estimate, block by block, gives
\[
  \sum_{M_r\le N}H_rM_r^{-\tau}
  =\sum_{j\le N}j^{-\tau}
   +O_{\tau}(N^{1-\lambda-\tau})+O(1),
\]
where the last block is taken with its actual length.  Its discrepancy error
is $O(N^{1-\sigma})$, which is no larger than the errors displayed above.
Because
\[
  1+\lambda-\sigma=1-\lambda-\tau
  =1-\frac{\sigma+\tau}{2},
\]
summing \cref{eq:moving-window-adaptive-block} and using
\[
  \sum_{j\le N}j^{-\tau}
  =\frac{N^{1-\tau}}{1-\tau}+O(1)
  \qquad(0<\tau<1)
\]
proves \cref{eq:two-subcritical-asymptotic}.

If $\tau=0$, the target is the fixed interval $(0,\kappa)$, so the sharper
claim follows directly from \cref{eq:ordinary-interval-discrepancy}.  Finally,
on $N\le j<2N$ the target is squeezed between fixed intervals of lengths
$\kappa(2N)^{-\tau}$ and $\kappa N^{-\tau}$.  Formula
\cref{eq:fixed-window-short-block} and $\tau<\sigma$ give the asserted
two-sided order on every sufficiently large dyadic block.
\end{proof}

\begin{remark}[The exact critical boundary]\label{rem:log-disc-barrier}
The prefix condition in layer $j$ is
$d_+(x_j)<z_j-x_j$, where
\[
  z_j-x_j=\frac1{j\log b}
  +O_{c,b}\left(\frac{\log j}{j^2}\right).
\]
It therefore has critical exponent $\tau=1$, whereas
\cref{thm:two-subcritical-moving} requires $\tau<\sigma\le1/2$.
Even the weighted estimate \cref{eq:log-interval-discrepancy}, applied on a
dyadic block to a fixed interval of length $\asymp1/N$, has main term
$\asymp1/N$ and error $O_{c,b,\sigma}(1)$.  The ordinary estimate has error
$O_{c,b,\sigma}(N^{1-\sigma})$ against a constant-size main term.  Thus neither
estimate forces a critical hit.  The new theorem proves deterministic
moving-target asymptotics strictly below the critical scale for every fixed
algebraic pair covered above, but remains compatible with both finiteness
and infinitude.  In particular, it does not decide $\Sset_{2,10}$.
\end{remark}

%% file: sections/en/04_counting_identity.tex
\section{Exact counting and signed discrepancy}

The total target mass telescopes; for every integer $N\ge1$,
\begin{equation}\label{eq:mass-telescope}
  \sum_{m\le N}\delta_b(m)=\log_b(N+1),\qquad
  \sum_{b^{j-1}\le m<b^j}\delta_b(m)=1.
\end{equation}
The divergence in \cref{eq:mass-telescope} motivates a logarithmic counting
law, but deterministic dependence prevents a direct Borel--Cantelli
argument at a fixed parameter.

The next identity isolates that dependence in one fixed interval.  From now
on in deterministic statements we assume $c\ge2$.

\begin{theorem}[Exact discrepancy identity]\label{thm:exact-discrepancy}
Assume $c\ge2$.  Write $\alpha=s+\rho$, where
$s=\lfloor\alpha\rfloor$ and
$\rho=\{\alpha\}$.  Define
\[
  y_n=n\alpha-\log_b(n+1)
\]
and
\[
  B_{\alpha,b}(N)
  =\#\{0\le n<N:\{y_n\}\in[1-\rho,1)\},
\]
with the interval interpreted as empty when $\rho=0$.  Then
\begin{equation}\label{eq:exact-discrepancy}
  \boxed{
  A_{c,b}(N)=\log_b(N+1)+B_{\alpha,b}(N)-\rho N+\{y_N\}.}
\end{equation}
Consequently, if $D_{\alpha,b}(N)=B_{\alpha,b}(N)-\rho N$, then
\begin{align*}
  \Sset_{c,b}\text{ is infinite}
  &\Longleftrightarrow
    \log_b(N+1)+D_{\alpha,b}(N)\text{ is unbounded},\\
  A_{c,b}(N)\sim\log_bN
  &\Longleftrightarrow D_{\alpha,b}(N)=o(\log N),\\
  \Sset_{c,b}\text{ is finite}
  &\Longleftrightarrow
    D_{\alpha,b}(N)=-\log_b(N+1)+O(1).
\end{align*}
\end{theorem}

\begin{proof}
Let $I_m=\ind_{\{m\in\Sset_{c,b}\}}$.  By
\cref{prop:criterion} and $0<\delta_b(m)\le1$,
\[
  I_m=\lfloor F_{c,b}(m)\rfloor
      -\lfloor F_{c,b}(m)-\delta_b(m)\rfloor.
\]
Now
\[
  F_{c,b}(m)=y_{m-1}+\alpha,\qquad
  F_{c,b}(m)-\delta_b(m)=y_m,
\]
and
\[
  \lfloor y_{m-1}+\alpha\rfloor
  =\lfloor y_{m-1}\rfloor+s
   +\ind_{\{\{y_{m-1}\}\in[1-\rho,1)\}}.
\]
Summing from $m=1$ to $N$ telescopes the integer parts and gives
\[
  A_{c,b}(N)=Ns+B_{\alpha,b}(N)-\lfloor y_N\rfloor.
\]
Substituting
\[
  \lfloor y_N\rfloor
  =Ns+N\rho-\log_b(N+1)-\{y_N\}
\]
proves \cref{eq:exact-discrepancy}.  The three equivalences follow because
$\{y_N\}$ is bounded and $A_{c,b}(N)$ is a nondecreasing integer-valued
function.
\end{proof}

\begin{remark}\label{rem:fixed-obstruction}
For $(c,b)=(2,10)$, ordinary uniform distribution yields only
$D_{\alpha,10}(N)=o(N)$.  Infinitude requires ruling out the much finer
cancellation
$D_{\alpha,10}(N)=-\log_{10}N+O(1)$.  This is one exact formulation of the
fixed-parameter obstruction.
\end{remark}

%% file: sections/en/05_arithmetic_structure.tex
\section{Arithmetic classification, resonances, and deterministic sparsity}

\begin{proposition}[Rational--transcendental dichotomy for integer pairs]
\label{prop:rational-transcendental-dichotomy}
Let $c,b\ge2$ be integers.  Then
\[
  \log_b c\in\mathbb Q
  \quad\Longleftrightarrow\quad
  c\text{ and }b\text{ are multiplicatively dependent}.
\]
If these equivalent conditions fail, then $\log_b c$ is transcendental;
in particular, it can never be algebraic irrational.
\end{proposition}

\begin{proof}
Writing $\log_b c=p/q$ is equivalent to $c^q=b^p$, which is precisely
multiplicative dependence after reducing $p/q$.  Conversely, any relation
$c^q=b^p$ gives the same rational logarithm.  If $\log_b c$ were algebraic
irrational, the Gelfond--Schneider theorem would make
$b^{\log_b c}=c$ transcendental, contradicting $c\in\mathbb Z$; see, for
example, \cite{baker1975transcendental}.
\end{proof}

\begin{theorem}[Effective rigidity of integer endpoint equations]
\label{thm:integer-endpoint-rigidity}
Let $c,b\ge2$ be multiplicatively independent integers.  For
$e\in\{0,1\}$, define
\[
  \mathcal Z_e(c,b)=
  \left\{
    m\ge1:c^m=(m+e)b^k
    \text{ for some }k\in\mathbb Z_{\ge0}
  \right\}.
\]
Then both $\mathcal Z_0(c,b)$ and $\mathcal Z_1(c,b)$ are effectively
finite.  More precisely:
\begin{enumerate}[label=\textup{(\alph*)}]
  \item If $\operatorname{rad}(c)\ne\operatorname{rad}(b)$, then no
  $m\ge2$ belongs to either endpoint set.  At $m=1$, the left endpoint is
  impossible under multiplicative independence, whereas the right endpoint
  occurs exactly when $c=2b^k$ for some $k\ge0$.

  \item Suppose $\operatorname{rad}(c)=\operatorname{rad}(b)$.  For every
  prime $r$ in their common support, put
  \[
    C_r=v_r(c),
    \qquad B_r=v_r(b).
  \]
  There exist two common prime divisors $r,s$ such that
  $\Delta_{r,s}=C_rB_s-C_sB_r\ne0$.  For any such pair define
  \begin{equation}\label{eq:integer-endpoint-A}
    A_{r,s}=
    \frac{B_s/\log r+B_r/\log s}{|\Delta_{r,s}|}.
  \end{equation}
  Every solution in either endpoint set satisfies
  \begin{equation}\label{eq:integer-endpoint-bound}
    \boxed{
    m\le4(A_{r,s}+1)\log(2A_{r,s}+2).}
  \end{equation}
\end{enumerate}
Thus factorization of $b$ and $c$ produces an explicit finite search
interval for every possible endpoint equality.
\end{theorem}

\begin{proof}
Assume first that a prime $r$ divides $c$ but not $b$.  An endpoint equality
would give
\[
  m v_r(c)=v_r(m+e).
\]
For $m\ge2$, the left side is at least $m$, whereas
$v_r(m+e)\le\log(m+1)/\log r<m$, a contradiction.  If instead a prime
$r$ divides $b$ but not $c$, then
\[
  0=v_r(m+e)+k v_r(b),
\]
so $k=0$ and $c^m=m+e$.  This is impossible for $m\ge2$, since
$c^m\ge2^m>m+1$.  These alternatives exhaust unequal prime supports.
For $m=1$, the equations are respectively $c=b^k$ and $c=2b^k$; the first
contradicts multiplicative independence, and the second gives the stated
exceptional right endpoint.

Now suppose that the prime supports agree.  If
$C_rB_s-C_sB_r=0$ for every pair $r,s$, the two integer exponent vectors
$(C_r)_r$ and $(B_r)_r$ are proportional over $\mathbb Q$.  Clearing
denominators would give positive integers $u,v$ with $c^u=b^v$, contrary
to multiplicative independence.  Hence a pair with
$\Delta_{r,s}\ne0$ exists.

Taking $r$- and $s$-adic valuations in $c^m=(m+e)b^k$ gives
\[
  mC_r=v_r(m+e)+kB_r,
  \qquad
  mC_s=v_s(m+e)+kB_s.
\]
Eliminating $k$ yields
\begin{equation}\label{eq:integer-endpoint-elimination}
  m\Delta_{r,s}
  =B_s v_r(m+e)-B_r v_s(m+e).
\end{equation}
Since $m+e\le m+1$, taking absolute values gives
$m\le A_{r,s}\log(m+1)$.

For completeness, put $A=A_{r,s}$ and
$T=4(A+1)\log(2A+2)$.  The function
$t/\log(t+1)$ is strictly increasing on $(0,\infty)$.  Writing
$y=2(A+1)$, one has $T=2y\log y$ and, for $y\ge2$,
\[
  T+1\le y^2+1\le y^3.
\]
Consequently,
\[
  A\log(T+1)\le3A\log y
  <\frac32y\log y<T.
\]
Thus $T/\log(T+1)>A$, and monotonicity shows that every
$m\le A\log(m+1)$ must satisfy $m\le T$.  This proves
\cref{eq:integer-endpoint-bound}.  Once $m$ is bounded, the unique possible
scale satisfies
\[
  k=\frac{m\log c-\log(m+e)}{\log b},
  \qquad 0\le k\le m\log_b c,
\]
so the remaining verification is finite and effective.
\end{proof}

\begin{corollary}[No nontrivial endpoint degeneracy for $(2,10)$]
\label{cor:two-endpoint-rigidity}
The two decimal endpoint equations satisfy
\[
  2^m=m10^k
  \quad\text{for no }m\ge1,k\ge0,
\]
and
\[
  2^m=(m+1)10^k
  \quad\Longleftrightarrow\quad (m,k)=(1,0).
\]
The latter is an excluded right endpoint, so $m=1$ is not a prefix hit.
For every $m\ge2$, any actual hit satisfies
\[
  m10^k<2^m<(m+1)10^k.
\]
\end{corollary}

\begin{proof}
Taking $5$-adic valuations in either endpoint equation forces $k=0$.
The first equation then becomes $2^m=m$, which has no positive solution.
The second becomes $2^m=m+1$; it holds at $m=1$, while
$2^m>m+1$ for every $m\ge2$.
\end{proof}

\begin{remark}
The effective-finiteness theorem is specific to integer $b,c$ and concerns
equality at the two endpoints only.  It must not be extended verbatim to a
nonintegral algebraic base, and it does not imply finiteness of the strict
prefix hit set in the multiplicatively independent branch.
\end{remark}

\subsection{Rational logarithmic slopes}

The rational-slope branch admits an exact parametrization even when $c$ is
not an integer.

\begin{theorem}[Exact rational-slope reduction]
\label{thm:rational-slope-reduction}
Assume
\[
  \log_b c=\frac PQ,
  \qquad P,Q\ge1,\qquad (P,Q)=1.
\]
Then
\begin{equation}\label{eq:rational-slope-parametrization}
  \Sset_{c,b}
  =\left\{
    \left\lfloor b^{n/Q}\right\rfloor:
    \begin{array}{l}
      n\ge0,\quad
      n\equiv P\lfloor b^{n/Q}\rfloor\pmod Q,\\
      P\lfloor b^{n/Q}\rfloor\ge n
    \end{array}
  \right\}.
\end{equation}
The final inequality is automatic for all sufficiently large $n$, and
\begin{equation}\label{eq:rational-slope-upper-count}
  A_{c,b}(X)\le Q\log_b(X+1)+O_{c,b}(1).
\end{equation}
Moreover,
\begin{equation}\label{eq:rational-slope-reciprocal}
  \sum_{m\in\Sset_{c,b}}\frac1m<\infty.
\end{equation}
If $Q=1$, then the reduction is completely explicit:
\begin{equation}\label{eq:integral-slope-exact-set}
  \Sset_{b^P,b}=\{b^n:n\ge0\},
  \qquad
  A_{b^P,b}(X)=\lfloor\log_bX\rfloor+1.
\end{equation}
Equivalently, writing $n=Q\ell+r$ with $0\le r<Q$, infinitude is reduced
to deciding whether, for some fixed $r$,
\begin{equation}\label{eq:rational-residue-orbit}
  P\left\lfloor b^\ell b^{r/Q}\right\rfloor
  \equiv r\pmod Q
\end{equation}
holds for infinitely many $\ell$.
\end{theorem}

\begin{proof}
If $m\in\Sset_{c,b}$, choose the corresponding $k\ge0$ and put
$n=Pm-Qk$.  Then
\[
  \frac{c^m}{b^k}=b^{n/Q}\in[m,m+1),
\]
so $n\ge0$, $m=\lfloor b^{n/Q}\rfloor$, and $n\equiv Pm\pmod Q$.
Moreover $k=(Pm-n)/Q\ge0$.  Conversely, every $n$ satisfying the
conditions in \cref{eq:rational-slope-parametrization}, with
$m=\lfloor b^{n/Q}\rfloor$ and $k=(Pm-n)/Q$, gives
\[
  c^m=b^k b^{n/Q}\in[mb^k,(m+1)b^k),
\]
hence a prefix hit.  Since $\lfloor b^{n/Q}\rfloor$ grows exponentially,
$Pm\ge n$ eventually.  If $m\le X$, then $b^{n/Q}<X+1$, which proves
\cref{eq:rational-slope-upper-count}.  For all sufficiently large $n$,
$\lfloor b^{n/Q}\rfloor\ge b^{n/Q}/2$; hence the reciprocal sum is bounded
by the convergent geometric series $2\sum_{n\ge0}b^{-n/Q}$.  When $Q=1$,
the congruence is void, $\lfloor b^n\rfloor=b^n$, and $Pb^n\ge n$ for every
$n\ge0$, proving \cref{eq:integral-slope-exact-set}.  Substitution of
$n=Q\ell+r$ gives \cref{eq:rational-residue-orbit}.
\end{proof}

\begin{theorem}[Complete criterion for integral radical branches]
\label{thm:integral-radical-branches}
Retain $P,Q$ from \cref{thm:rational-slope-reduction}.  Write
\[
  b=\prod_{p\mid b}p^{e_p},
  \qquad g=\gcd_{p\mid b}e_p,
\]
and factor
\[
  Q_0=\prod_{p\mid b}p^{v_p(Q)},
  \qquad Q_1=Q/Q_0,
  \qquad (Q_1,b)=1.
\]
Put $T=\operatorname{ord}_{Q_1}(b)$ when $Q_1>1$, and $T=1$ when
$Q_1=1$.  Fix $0\le r<Q$ such that $Q\mid gr$, and set
$M_r=b^{r/Q}\in\mathbb N$.  The branch $n=Q\ell+r$ contains infinitely
many solutions if and only if
\begin{equation}\label{eq:integral-radical-criterion}
  \boxed{
  Q_0\mid r,
  \qquad
  r(PM_r)^{-1}\pmod{Q_1}
  \in\langle b\rangle
  \subset(\mathbb Z/Q_1\mathbb Z)^\times.}
\end{equation}
For $Q_1=1$ the second condition is void.  Whenever
\cref{eq:integral-radical-criterion} holds, all sufficiently large solutions
on this branch are
\[
  m=M_rb^\ell,
  \qquad \ell\equiv\ell_r\pmod T,
\]
for one residue class $\ell_r$, and the branch contributes
\begin{equation}\label{eq:integral-radical-count}
  \frac1T\log_bX+O_{b,P,Q}(1)
\end{equation}
solutions not exceeding $X$.
\end{theorem}

\begin{proof}
Valuations show that $b^{r/Q}$ is rational exactly when
$Q\mid e_pr$ for every $p\mid b$, equivalently $Q\mid gr$; being a rational
algebraic integer, it is then an integer.  On the branch $n=Q\ell+r$ the
parametrization becomes
\[
  m=M_rb^\ell,
  \qquad PM_rb^\ell\equiv r\pmod Q.
\]
For sufficiently large $\ell$, the left side vanishes modulo $Q_0$, so the
$Q_0$ component is soluble exactly when $Q_0\mid r$.  Modulo $Q_1$, all of
$P,M_r,b$ are units.  In particular
$\gcd(PM_r,Q_1)=1$, so $(PM_r)^{-1}\pmod {Q_1}$ exists, and the congruence
is equivalent to
\[
  b^\ell\equiv r(PM_r)^{-1}\pmod{Q_1}.
\]
It is soluble precisely under the subgroup condition in
\cref{eq:integral-radical-criterion}; if soluble, its solutions form one
class modulo $T$.  The remaining inequality $PM_rb^\ell\ge Q\ell+r$ is
automatic eventually.  Counting the resulting geometric progression proves
\cref{eq:integral-radical-count}.
\end{proof}

\begin{corollary}[A general rational-slope infinitude criterion]
\label{cor:rational-slope-smooth-denominator}
If every prime divisor of $Q$ divides $b$, equivalently
$\operatorname{rad}(Q)\mid b$, then $\Sset_{c,b}$ is infinite.  In fact,
every sufficiently large $m=b^\ell$ is a solution, and
\[
  \log_b X+O_{c,b}(1)
  \le A_{c,b}(X)
  \le Q\log_b(X+1)+O_{c,b}(1).
\]
\end{corollary}

\begin{proof}
The hypothesis gives $Q_1=1$.  Apply
\cref{thm:integral-radical-branches} with $r=0$ and $M_0=1$; equivalently,
$Q\mid b^\ell$ for all sufficiently large $\ell$.  The lower bound follows
from that geometric branch, and the upper bound is
\cref{eq:rational-slope-upper-count}.
\end{proof}

\begin{corollary}[An integral branch beyond the radical criterion]
\label{cor:radical-criterion-strict-example}
The condition $\operatorname{rad}(Q)\mid b$ is not necessary.  For example,
\[
  b=9,
  \qquad c=3^{1/3}=9^{1/6}
\]
has $\operatorname{rad}(6)=6\nmid9$, but
\begin{equation}\label{eq:ninth-root-explicit-family}
  \boxed{3\cdot9^\ell\in\Sset_{3^{1/3},9}\qquad(\ell\ge0).}
\end{equation}
\end{corollary}

\begin{proof}
Here $P=1,Q=6$, $g=2$, $Q_0=3$, and $Q_1=2$.  Taking $r=3$ gives
$M_r=9^{1/2}=3$, and $3\cdot9^\ell\equiv3\pmod6$ for every $\ell$.
Alternatively, for
\[
  m=3\cdot9^\ell,
  \qquad k=\frac{9^\ell-(2\ell+1)}2\in\mathbb Z_{\ge0},
\]
one has $(3^{1/3})^m=3^{9^\ell}=m9^k$, which is an allowed left-endpoint
prefix hit.
\end{proof}

\begin{remark}
Theorem~\ref{thm:integral-radical-branches} exhausts every residue branch for
which $b^{r/Q}$ is an integer.  On a remaining nonintegral branch,
\cref{eq:rational-residue-orbit} can encode the orbit of a specified
algebraic irrational under multiplication by $b$.  No general normality or
density theorem is available for such a prescribed algebraic number, so the
reduction alone does not justify a blanket finiteness or infinitude claim.
The next subsection gives a complete answer in the more rigid case where
both $c$ and $b$ are integers and multiplicatively dependent.
\end{remark}

\subsection{The multiplicatively dependent integer case}

\begin{theorem}[Complete dependent classification]\label{thm:dependent}
Let $c,b\ge2$ be integers and suppose
\[
  \log_b c=\frac pq,\qquad p,q\ge1,\qquad (p,q)=1.
\]
There is a unique integer $d\ge2$ such that $c=d^p$ and $b=d^q$, and
\begin{equation}\label{eq:dependent-set}
  \boxed{
  \Sset_{c,b}=\{d^n:n\ge0,\ n\equiv p d^n\pmod q\}.}
\end{equation}
Moreover, the congruence in \cref{eq:dependent-set} is eventually periodic
and has exactly a proportion $1/q$ of solutions in every full eventual
period.  In particular,
\begin{equation}\label{eq:dependent-count}
  A_{c,b}(X)=\frac{\log X}{q\log d}+O_{c,b}(1)
  =\log_bX+O_{c,b}(1),
\end{equation}
so $\Sset_{c,b}$ is infinite.
\end{theorem}

\begin{proof}
The identity $c^q=b^p$ and unique factorization imply
$c=d^p$, $b=d^q$ for a unique integer $d\ge2$.  If
$m\in\Sset_{c,b}$, then for some $k\in\mathbb Z_{\ge0}$,
\[
  m\le d^{pm-qk}<m+1.
\]
Put $n=pm-qk\in\mathbb Z$.  The lower inequality forces $n\ge0$; hence the
middle term is an integer and must equal $m$.  Thus $m=d^n$ and
$n\equiv pd^n\pmod q$.  Conversely, this congruence makes
$k=(pd^n-n)/q$ an integer.  Since $d^n\ge n$ and $p\ge1$, it is
nonnegative, and $c^{d^n}=d^n b^k$.

For the periodic statement, define
\[
  q_{\rm nil}=\prod_{\substack{\ell^e\parallel q\\ \ell\mid d}}\ell^e,
  \qquad q_{\rm unit}=q/q_{\rm nil},
\]
and let $\nu$ be the least integer such that
$d^n\equiv0\pmod{q_{\rm nil}}$ for every $n\ge\nu$.  Put
\[
  t=\begin{cases}
    1,&q_{\rm unit}=1,\\
    \operatorname{ord}_{q_{\rm unit}}(d),&q_{\rm unit}>1,
  \end{cases}
  \qquad P=\operatorname{lcm}(q,t).
\]
Then $H(n)=n-pd^n\pmod q$ has period $P$ for $n\ge\nu$.  We prove by
strong induction on $q$ that, for every $N\ge\nu$ and every $y\pmod q$,
\[
  \#\{N\le n<N+P:H(n)\equiv y\pmod q\}=P/q.
\]
The assertion is trivial for $q=1$.  If $q>1$, then $t<q$: this is clear
when $q_{\rm unit}=1$, and otherwise
$t\le\varphi(q_{\rm unit})<q_{\rm unit}\le q$.  Let
$g=(q,t)<q$.  Every integer in $[N,N+P)$ is represented uniquely as
\[
  n=N+r+jt,\qquad 0\le r<t,\qquad 0\le j<q/g.
\]
For fixed $r$, nilpotence modulo $q_{\rm nil}$ and periodicity modulo
$q_{\rm unit}$ give
\[
  H(N+r+jt)\equiv H(N+r)+jt\pmod q.
\]
As $j$ varies, $jt$ runs once through the subgroup
$g\mathbb Z/q\mathbb Z$.  Hence the congruence to $y$ has exactly one
solution $j$ if $H(N+r)\equiv y\pmod g$, and none otherwise.

Apply the induction hypothesis to the same expression reduced modulo $g$.
To justify the period used here, note first that $g=(q,t)$ divides $t$.
The unit part of $g$ divides the unit part $q_{\rm unit}$ of $q$, so the
multiplicative order of $d$ modulo that unit part divides
$t=\operatorname{ord}_{q_{\rm unit}}(d)$ (with the evident convention when
$q_{\rm unit}=1$).  The nilpotent part of $g$ also divides $g\mid t$.
Thus the eventual period supplied by the induction hypothesis may be taken
as
\[
  P_g=\operatorname{lcm}(g,t_g),
  \qquad t_g\mid t,
\]
and indeed $P_g\mid t$.  Its nilpotence threshold is at most $\nu$.
Consequently, among the $t$
successive values $N+r$ exactly $t/g$ satisfy
$H(N+r)\equiv y\pmod g$.  Thus exactly $t/g=P/q$ pairs $(r,j)$ contribute,
completing the induction.  In particular, zero occurs exactly $P/q$ times
per eventual period.  Counting those $n\le\log_dX$, with the finite initial
segment absorbed into $O_{c,b}(1)$, proves \cref{eq:dependent-count}.
\end{proof}

\subsection{The independent case}

\begin{proposition}[Uniform distribution]\label{prop:ud}
If $\alpha=\log_b c$ is irrational, then
\[
  \bigl\{m\alpha-\log_bm\bigr\}_{m\ge1}
\]
is uniformly distributed modulo one.  Hence $\Sset_{c,b}$ has natural
density and logarithmic density zero.
\end{proposition}

\begin{proof}
Fix $h\in\mathbb Z\setminus\{0\}$ and set
$z=\exp(2\pi i h\alpha)\ne1$ and
$\tau=2\pi h/\log b$.  Abel summation, applied to the bounded geometric
partial sums of $z^m$ and the weight $m^{-i\tau}$, gives
\[
  \sum_{m\le N}\exp\bigl(2\pi i hF_{c,b}(m)\bigr)=O_{h,c,b}(\log N)=o(N).
\]
Weyl's criterion~\cite{kuipersniederreiter1974} proves uniform distribution.
Since the target lengths $\delta_b(m)$ tend to zero, every sufficiently
large solution lies in any prescribed fixed interval $[0,\varepsilon)$;
letting $\varepsilon\downarrow0$ gives zero natural density.  Partial
summation then gives zero logarithmic density: if
$A(t)=\#(\Sset_{c,b}\cap[1,t])=o(t)$, then
\[
  \sum_{\substack{m\le X\\m\in\Sset_{c,b}}}\frac1m
  =\frac{A(X)}X+\int_1^X\frac{A(t)}{t^2}\,dt
  =o(\log X).
\]
\end{proof}

\begin{theorem}[Fixed differences]\label{thm:fixed-difference}
Assume $\alpha$ is irrational.  For every fixed $h\ge1$, put
\[
  \beta_h=\{h\alpha\},\qquad
  B_h=\frac{h+1}{b^{\beta_h}-1}.
\]
If $m,m+h\in\Sset_{c,b}$, then $m<B_h$.  Thus every fixed nonzero
difference occurs only finitely often; in particular, if the solution set is
infinite, its consecutive gaps tend to infinity.
\end{theorem}

\begin{proof}
For a solution $u$, write
$F_{c,b}(u)=j_u+\theta_u$ with
$0\le\theta_u<\delta_b(u)$.  Subtracting the relations for $m+h$ and $m$
gives
\[
  h\alpha-(j_{m+h}-j_m)
  =\log_b\left(1+\frac hm\right)+\theta_{m+h}-\theta_m.
\]
The right-hand side is positive and is strictly less than
$\log_b(1+(h+1)/m)$.  The left-hand side is congruent to $\beta_h$ modulo
one, and positivity implies that it is at least $\beta_h$.  Therefore
$\beta_h<\log_b(1+(h+1)/m)$, which is equivalent to $m<B_h$.
\end{proof}

\begin{theorem}[Exact double-hit resonance shells]
\label{thm:resonance-shell}
Assume $\alpha=\log_b c$ is irrational and fix $h\ge1$.  Put
\[
  \beta_h=\{h\alpha\}\in(0,1),
  \qquad t_{h,k}=b^{\beta_h+k}
  \quad(k\in\mathbb Z_{\ge0}).
\]
If $m,m+h\in\Sset_{c,b}$, write
$j_u=\lfloor F_{c,b}(u)\rfloor$ for $u=m,m+h$.  Then
\[
  k=h\alpha-(j_{m+h}-j_m)-\beta_h
\]
is a nonnegative integer and satisfies
\begin{equation}\label{eq:resonance-shell-range}
  t_{h,k}<h+2
\end{equation}
and
\begin{equation}\label{eq:resonance-shell}
  \boxed{
  \frac{h-t_{h,k}}{t_{h,k}-1}
  <m<
  \frac{h+1}{t_{h,k}-1}.}
\end{equation}
The real length of this open shell is exactly
\begin{equation}\label{eq:resonance-shell-width}
  \boxed{
  W_{h,k}=\frac{t_{h,k}+1}{t_{h,k}-1}
  =1+\frac2{t_{h,k}-1}.}
\end{equation}
If, in addition, $h\le(b-1)m-1$, then necessarily $k=0$, and hence
\begin{equation}\label{eq:principal-resonance-shell}
  \boxed{
  \frac{h-b^{\beta_h}}{b^{\beta_h}-1}
  <m<
  \frac{h+1}{b^{\beta_h}-1}.}
\end{equation}
\end{theorem}

\begin{proof}
For the two hits write
\[
  F_{c,b}(m)=j_m+\theta_m,
  \qquad
  F_{c,b}(m+h)=j_{m+h}+\theta_{m+h},
\]
where
$0\le\theta_m<\delta_b(m)$ and
$0\le\theta_{m+h}<\delta_b(m+h)$.  Set
\[
  R=\log_b\left(1+\frac hm\right),
  \qquad
  x=h\alpha-(j_{m+h}-j_m).
\]
Subtraction gives $x=R+\theta_{m+h}-\theta_m$.  Keeping both strict
half-open endpoints yields
\begin{equation}\label{eq:resonance-two-sided-log}
  \log_b\frac{m+h}{m+1}
  <x<
  \log_b\frac{m+h+1}{m}.
\end{equation}
The left endpoint is nonnegative and the inequality is strict, so $x>0$.
Moreover $x\equiv\beta_h\pmod1$; consequently there is a unique $k\ge0$
with $x=\beta_h+k$.  Exponentiating
\cref{eq:resonance-two-sided-log} gives
\[
  \frac{m+h}{m+1}<t_{h,k}<\frac{m+h+1}{m}.
\]
Solving the two inequalities for $m$ proves
\cref{eq:resonance-shell}.  Since $m\ge1$, the last right-hand side is at
most $h+2$, proving \cref{eq:resonance-shell-range}; subtracting the shell
endpoints gives \cref{eq:resonance-shell-width}.  Finally,
$h\le(b-1)m-1$ implies $(m+h+1)/m\le b$, hence $t_{h,k}<b$ and $k=0$.
\end{proof}

The old upper bound in \cref{thm:fixed-difference} has scale
$h/\beta_h$ when $\beta_h$ is small.  The principal shell
\cref{eq:principal-resonance-shell} has scale $1/\beta_h$ in width and lies
near the far endpoint of that range, saving a full factor of $h$.

\begin{corollary}[Separation of resonance shells for $b\ge3$]
\label{cor:shell-separation-bge3}
Assume $b\ge3$ and define
\[
  I_{h,k}=\left(
  \frac{h-t_{h,k}}{t_{h,k}-1},
  \frac{h+1}{t_{h,k}-1}
  \right),
  \qquad
  K_h=\{k\ge0:t_{h,k}<h+2\}.
\]
The set $K_h$ is finite.  Whenever $k,k+1\in K_h$, the shell
$I_{h,k+1}$ lies strictly to the left of $I_{h,k}$, with a nonempty gap
between them.  Thus all starting points for difference $h$ lie in at most
\[
  \max\left\{0,
  \left\lceil\log_b(h+2)-\beta_h\right\rceil\right\}
\]
pairwise disjoint open shells.
\end{corollary}

\begin{proof}
Put $t=t_{h,k}$.  If $k+1\in K_h$, then $bt<h+2$, and direct subtraction
gives
\[
  \inf I_{h,k}-\sup I_{h,k+1}
  =\frac{t((b-1)h-bt)+1}{(t-1)(bt-1)}.
\]
If $h=1$, the condition $bt<h+2=3$ is impossible because $b\ge3$ and
$t>1$.  If $h\ge2$, then
$(b-1)h-bt>(b-2)h-2\ge0$, so the displayed gap is positive.
The count is the number of integers
$0\le k<\log_b(h+2)-\beta_h$.
\end{proof}

\begin{corollary}[A global elementary gap for $\Sset_{2,10}$]
\label{cor:two-global-gap-four}
Distinct elements of $\Sset_{2,10}$ differ by at least $4$.
\end{corollary}

\begin{proof}
For $h=1,2,3$ one has
$10^{\{h\log_{10}2\}}=2,4,8$, respectively.  The bound in
\cref{thm:fixed-difference} is then $2,1,4/7$.  Differences $2$ and $3$
are impossible.  Difference $1$ could only start at $m=1$, but $2^1$ does
not begin with $1$.
\end{proof}

\begin{corollary}[Fixed-difference representation bound for $2^m$]
\label{cor:two-difference-representation}
Let
\[
  r_2(h)=\#\{m\ge1:m,m+h\in\Sset_{2,10}\}.
\]
With $K_h$ and $W_{h,k}$ as above,
\begin{equation}\label{eq:two-difference-representation}
  \boxed{
  r_2(h)\le
  \sum_{k\in K_h}\left\lceil\frac{W_{h,k}}4\right\rceil.}
\end{equation}
Every $k\ge1$ shell contributes at most one representation, and
\begin{equation}\label{eq:two-difference-simplified}
  r_2(h)
  \le \#(K_h\setminus\{0\})+
  \left\lceil\frac14\left(1+\frac2{10^{\beta_h}-1}\right)\right\rceil.
\end{equation}
Consequently
\[
  r_2(h)=O\left(\log(h+2)+\frac1{\{h\log_{10}2\}}\right)
\]
with an absolute implied constant.
\end{corollary}

\begin{proof}
Starting points belonging to the same shell are elements of
$\Sset_{2,10}$ and hence, by \cref{cor:two-global-gap-four}, are separated
by at least $4$.  An open interval of length $W$ contains at most
$\lceil W/4\rceil$ such points, proving
\cref{eq:two-difference-representation}.  For $k\ge1$ one has
$t_{h,k}>10$ and therefore $W_{h,k}<11/9<4$, which proves
\cref{eq:two-difference-simplified}.  Finally use
$10^x-1\ge x\log10$ and $\#K_h=O(\log(h+2))$.
\end{proof}

\begin{definition}
For an irrational number $\xi$, its irrationality exponent is
\[
  \mu(\xi)=\inf\left\{\nu:
  \begin{array}{c}
    \left|\xi-p/q\right|<q^{-\nu}\text{ holds for only}\\[-1mm]
    \text{finitely many rational numbers }p/q
  \end{array}
  \right\}.
\]
\end{definition}

\begin{theorem}[Power-saving sparsity]\label{thm:power-saving}
Suppose $\mu(\alpha)<\infty$.  For every $\nu>\mu(\alpha)$ there is
$C_{c,b,\nu}>0$ such that any two solutions $m<m+h$ satisfy
\[
  h\ge C_{c,b,\nu}m^{1/\nu}.
\]
Consequently,
\begin{equation}\label{eq:power-saving}
  A_{c,b}(X)\ll_{c,b,\nu}X^{1-1/\nu},
  \qquad
  \sum_{m\in\Sset_{c,b}}\frac1m<\infty.
\end{equation}
If $b\ge2$ is an integer, $c>1$ is a positive real algebraic number, and
$\log_b c$ is irrational, a finite effective choice of $\nu$ follows from
lower bounds for linear forms in logarithms.
\end{theorem}

\begin{proof}
For $\nu>\mu(\alpha)$, the definition of the irrationality exponent gives
$\|h\alpha\|\ge c'_\nu h^{1-\nu}$ for all sufficiently large $h$.
Shrinking $c'_\nu$ to the minimum of the finitely many positive quantities
$\|h\alpha\|h^{\nu-1}$ in the omitted initial range gives a constant
$c_\nu>0$ for which the same inequality holds for every $h\ge1$.  Since
$\beta_h\ge\|h\alpha\|$ and
$b^{\beta_h}-1\ge(\log b)\beta_h$, \cref{thm:fixed-difference} implies
$m\ll_{c,b,\nu}h^\nu$.  This gives the gap bound.  In a dyadic interval
$[M,2M)$, consecutive solutions are separated by
$\gg M^{1/\nu}$, so their number is $O(M^{1-1/\nu})$.  Summing over dyadic
intervals proves the first estimate in \cref{eq:power-saving}; partial
summation proves convergence of the reciprocal sum.

Under the algebraic hypothesis in the theorem, choose $r$ nearest to
$h\alpha$.  Then $|r|\ll_{c,b}h$, and the linear form
$h\log c-r\log b$ has an effective lower bound of the shape
$\gg_{c,b}h^{-C}$ by Matveev's theorem~\cite{matveev2000}.  It is nonzero,
since equality would make $\log_b c=r/h$ rational.  This gives a finite
effective irrationality exponent.
\end{proof}

\subsection{Quantitative resonance chains}

The shell theorem also controls arithmetic progressions of actual hits that
do not start at a floor resonance center.  After at most one exceptional
edge, every such progression is forced into the principal shell.

\begin{theorem}[Tail rigidity and maximal arithmetic hit chains]
\label{thm:arithmetic-hit-chain}
Assume $\alpha=\log_b c$ is irrational.  For $q\ge1$, put
\[
  a_q=\lfloor q\alpha\rfloor,
  \qquad \beta_q=\{q\alpha\},
  \qquad t_q=b^{\beta_q},
\]
and define
\[
  I_q=\left(\frac{q-t_q}{t_q-1},
             \frac{q+1}{t_q-1}\right),
  \qquad
  W_q=|I_q|=\frac{t_q+1}{t_q-1}.
\]
Suppose
\[
  n_j=m+jq\in\Sset_{c,b}
  \qquad(0\le j\le\ell-1),
\]
and let $K_j=\lfloor F_{c,b}(n_j)\rfloor$ be the corresponding radix
scale.  Then
\begin{equation}\label{eq:chain-tail-principal}
  n_j\in I_q,
  \qquad K_{j+1}-K_j=a_q
  \qquad(1\le j\le\ell-2).
\end{equation}
If $\ell\ge3$, then more precisely
\begin{equation}\label{eq:chain-start-window}
  \frac{q-t_q}{t_q-1}<m+q<
  \frac{q+1}{t_q-1}-(\ell-3)q,
\end{equation}
and hence
\begin{equation}\label{eq:chain-length-necessary}
  (\ell-3)q<W_q.
\end{equation}
Consequently the maximal length
\[
  \Lambda_{c,b}(q)=
  \sup\bigl(\{0\}\cup
  \{\ell:m,m+q,\ldots,m+(\ell-1)q\in\Sset_{c,b}
  \text{ for some }m\}\bigr)
\]
satisfies
\begin{equation}\label{eq:maximal-chain-bound}
  \boxed{
  \Lambda_{c,b}(q)
  \le2+\left\lceil\frac{W_q}{q}\right\rceil.}
\end{equation}
If the first term also satisfies $q\le(b-1)m-1$, then the first edge is
principal as well, and the right-hand side in
\cref{eq:maximal-chain-bound} improves to
$1+\lceil W_q/q\rceil$.

Every chain with $\ell\ge4$ and $(\ell-3)q>1$ forces the one-sided
approximation
\begin{equation}\label{eq:long-chain-approximation}
  \boxed{
  \{q\alpha\}<
  \log_b\left(1+\frac2{(\ell-3)q-1}\right).}
\end{equation}
If $\mu(\alpha)<\infty$, then for every $\nu>\mu(\alpha)$,
\begin{equation}\label{eq:maximal-chain-diophantine}
  \boxed{\Lambda_{c,b}(q)\ll_{c,b,\nu}q^{\nu-2}.}
\end{equation}
In particular, $\mu(\alpha)=2$ gives
$\Lambda_{c,b}(q)\ll_{c,b,\varepsilon}q^\varepsilon$.  If $\alpha$ is
badly approximable, then $\sup_q\Lambda_{c,b}(q)<\infty$.
\end{theorem}

\begin{proof}
Apply \cref{thm:resonance-shell} to each adjacent pair
$n_j,n_{j+1}$.  Its shell index is
\[
  k_j=a_q-(K_{j+1}-K_j)\ge0.
\]
For $j\ge1$ one has $n_j=m+jq\ge q+1$, and therefore
\[
  (b-1)n_j-1\ge n_j-1\ge q.
\]
The principal-shell clause of \cref{thm:resonance-shell} now gives
$k_j=0$, which proves \cref{eq:chain-tail-principal}.

When $\ell\ge3$, all points
$n_1,\ldots,n_{\ell-2}$ lie in the open interval $I_q$.  Since
$n_{\ell-2}=n_1+(\ell-3)q$, their two extreme inequalities give
\cref{eq:chain-start-window} and \cref{eq:chain-length-necessary}.  If
$r=\ell-3$ is a nonnegative integer and $r<W_q/q$, then
$r\le\lceil W_q/q\rceil-1$, which proves
\cref{eq:maximal-chain-bound}; the case $\ell=2$ is automatic.  Under the
extra hypothesis at $m$, all of $n_0,\ldots,n_{\ell-2}$ lie in $I_q$.
Their span is $(\ell-2)q<W_q$, giving the improved constant.

Since $W_q=1+2/(t_q-1)$, the inequality
$R=(\ell-3)q<W_q$ with $R>1$ implies
$t_q-1<2/(R-1)$.  Taking logarithms proves
\cref{eq:long-chain-approximation}.  Also
\[
  \Lambda_{c,b}(q)
  \le3+\frac{b+1}{q\log b\,\beta_q}
  \le3+\frac{b+1}{q\log b\,\|q\alpha\|}.
\]
For $\nu>\mu(\alpha)$ one has
$\|q\alpha\|\gg_{\alpha,\nu}q^{1-\nu}$, proving
\cref{eq:maximal-chain-diophantine}.  The final two assertions follow from
this estimate and, in the badly approximable case, from
$\|q\alpha\|\gg1/q$.
\end{proof}

The terminology of one-sided approximation used here is standard; see, for
example, Han\v{c}l and Turek~\cite{hanclturek2019}.  The content of this
corollary is that a self-prefix hit chain forces the sharp inequality needed
to invoke the classical Legendre criterion, rather than a new proof of that
criterion.

\begin{corollary}[Long hit chains have a convergent direction]
\label{cor:long-chain-convergent-direction}
Retain the notation of \cref{thm:arithmetic-hit-chain}, and put
\[
  C_b=\frac{\sqrt b+1}{\sqrt b-1}.
\]
If
\[
  m,m+q,\ldots,m+(\ell-1)q\in\Sset_{c,b},
  \qquad \ell\ge3+\lceil C_b\rceil,
\]
set
\[
  p=\lfloor q\alpha\rfloor,
  \qquad g=\gcd(p,q),
  \qquad a=p/g,
  \qquad d=q/g,
\]
where $\gcd(0,q)=q$.  Then $a/d$ is a regular continued-fraction
convergent of $\alpha$, approached from below:
\begin{equation}\label{eq:long-chain-convergent-direction}
  \boxed{d\alpha-a>0.}
\end{equation}
Equivalently, the rational number $p/q$, written in lowest terms, is a
positive-error convergent.  In particular, if $\gcd(p,q)=1$, then $q$
itself is a convergent denominator.
\end{corollary}

\begin{proof}
Put $L=\ell-3$, $\beta=\{q\alpha\}$ and $t=b^\beta$.  By
\cref{eq:chain-length-necessary}, and because $L\ge C_b>1$,
\[
  t<\frac{Lq+1}{Lq-1}.
\]
The power series for $\operatorname{arctanh}$ gives
\[
  q\operatorname{arctanh}\frac1{Lq}
  \le \operatorname{arctanh}\frac1L.
\]
Consequently,
\[
  \left(\frac{Lq+1}{Lq-1}\right)^{2q}
  \le
  \left(\frac{L+1}{L-1}\right)^2
  \le b,
\]
where the last inequality is equivalent to $L\ge C_b$.  The strict first
bound on $t$ therefore yields $t^{2q}<b$, or
\[
  0<\alpha-\frac pq=\frac\beta q<\frac1{2q^2}.
\]
After reducing $p/q=a/d$, one has $d\le q$, so the right side is less than
$1/(2d^2)$.  Legendre's theorem shows that $a/d$ is a regular
  continued-fraction convergent~\cite{khinchin1997}.  Its error is positive because
$d\alpha-a=\beta/g>0$.
\end{proof}

\begin{remark}[The exceptional first edge is genuine]
\label{rem:chain-first-edge-sharp}
For $(c,b)=(3,2)$, exact integer arithmetic gives
\[
  27,\ 95,\ 163\in\Sset_{3,2}
\]
with radix scales $38,144,251$.  Thus the step is $q=68$, while the two
successive scale differences are $106$ and $107$, and
$\lfloor68\log_2 3\rfloor=107$.  The first edge has shell index $1$ and
the second has index $0$.  Moreover $W_{68}<68$, so
\cref{eq:maximal-chain-bound} gives the attained upper bound $3$.  Hence
the exceptional first edge and the additive constant $2$ in the general
bound cannot in general be removed.
\end{remark}

\begin{proposition}[Exact shell sieve for a fixed difference]
\label{prop:fixed-difference-exact-sieve}
Assume $\alpha=\log_b c$ is irrational.
Let
\[
  r_{c,b}(h)=\#\{m\ge1:m,m+h\in\Sset_{c,b}\}.
\]
For $h\ge1$ retain $t_{h,k}=b^{\beta_h+k}$ and put
\[
  L_{h,k}=\frac{h-t_{h,k}}{t_{h,k}-1},
  \qquad
  U_{h,k}=\frac{h+1}{t_{h,k}-1},
  \qquad
  K_h=\{k\ge0:t_{h,k}<h+2\}.
\]
The number of positive integers in the open shell
$(L_{h,k},U_{h,k})$ is exactly
\begin{equation}\label{eq:fixed-difference-shell-count}
  N_{h,k}=
  \max\left\{0,
  \lceil U_{h,k}\rceil-
  \max\{1,\lfloor L_{h,k}\rfloor+1\}\right\}.
\end{equation}
Therefore
\begin{equation}\label{eq:fixed-difference-exact-sieve}
  \boxed{r_{c,b}(h)\le\sum_{k\in K_h}N_{h,k}},
\end{equation}
and, without any Diophantine hypothesis,
\begin{equation}\label{eq:fixed-difference-phase-bound}
  r_{c,b}(h)\ll_b
  \log(h+2)+\frac1{\{h\alpha\}}.
\end{equation}
If $\mu(\alpha)<\infty$, then for every $\nu>\mu(\alpha)$,
\begin{equation}\label{eq:fixed-difference-power-bound}
  \boxed{r_{c,b}(h)\ll_{c,b,\nu}h^{\nu-2}.}
\end{equation}
If $\alpha$ is badly approximable, then
$r_{c,b}(h)\ll_{c,b}\log(h+2)$.
\end{proposition}

\begin{proof}
The formula \cref{eq:fixed-difference-shell-count} is the exact count of
integers $m\ge1$ satisfying $L_{h,k}<m<U_{h,k}$.  Every double hit belongs
to a unique shell by \cref{thm:resonance-shell}, proving
\cref{eq:fixed-difference-exact-sieve}.  The principal shell contributes
$O_b(1+1/\beta_h)$ integer candidates.  If $k\ge1$, then
$t_{h,k}>b\ge2$, so its width is less than $3$; and the number of eligible
nonprincipal shells is $O_b(\log(h+2))$.  This proves
\cref{eq:fixed-difference-phase-bound}.

We now use actual separation between solutions to improve the worst-case
power by one.  Fix $\nu>\mu(\alpha)$.  By
\cref{thm:power-saving}, two hits $u<v$ satisfy
\begin{equation}\label{eq:solution-gap-used-for-packing}
  v-u\gg_{c,b,\nu}u^{1/\nu}.
\end{equation}
The nonprincipal shells contribute only $O_b(\log(h+2))$.  For the
principal shell set $t=b^{\beta_h}$,
\[
  L=\frac{h-t}{t-1},
  \qquad W=\frac{t+1}{t-1}.
\]
For $h\ge2b$, one has $L\gg_b h/(t-1)>0$.  If this shell contains
$R$ representation starts, then \cref{eq:solution-gap-used-for-packing}
and its width give
\[
  R\ll_{c,b,\nu}1+WL^{-1/\nu}
  \ll_{c,b,\nu}
  1+h^{-1/\nu}(t-1)^{-(\nu-1)/\nu}.
\]
But
\[
  t-1\ge(\log b)\beta_h
  \ge(\log b)\|h\alpha\|
  \gg_{\alpha,\nu}h^{1-\nu}.
\]
Substitution yields $R\ll h^{\nu-2}$.  Since $\nu>2$, the logarithmic
nonprincipal contribution and the finite range $h<2b$ are absorbed into
the same bound, proving \cref{eq:fixed-difference-power-bound}.

If $\alpha$ is badly approximable, the same argument uses
$\|h\alpha\|\gg1/h$ and the corresponding square-root solution-gap bound.
The principal-shell packing is then $O(1)$, leaving only the
$O_b(\log(h+2))$ nonprincipal shells.
\end{proof}

\subsection{Nested windows at floor resonance centers}

The preceding shell theorem localizes every double hit with a prescribed
difference.  At the floor of the corresponding real resonance center, the
two prefix windows satisfy a stronger exact containment.

\begin{theorem}[Exact nesting at a floor resonance center]
\label{thm:floor-resonance-nesting}
Let $c,b\ge2$ be integers, put $\alpha=\log_b c$, and suppose that
\[
  q\alpha-p=\varepsilon>0,
  \qquad q\in\mathbb Z_{\ge1},\quad p\in\mathbb Z.
\]
Define
\begin{equation}\label{eq:floor-resonance-data}
  \varrho=\frac{c^q}{b^p}=b^\varepsilon>1,
  \qquad
  x=\frac q{\varrho-1},
  \qquad
  M=\lfloor x\rfloor,
\end{equation}
and assume $M\ge1$.  Then necessarily $p\ge0$, and the half-open intervals
\[
  J_0=[M,M+1),
  \qquad
  J_1=\left[
    \frac{M+q}{\varrho},
    \frac{M+q+1}{\varrho}
  \right)
\]
satisfy
\begin{equation}\label{eq:floor-resonance-containment}
  \boxed{J_1\subseteq J_0},
  \qquad
  |J_1|=\varrho^{-1}.
\end{equation}
For every $K\in\mathbb Z$ one has the exact equivalence
\begin{equation}\label{eq:floor-resonance-scale-equivalence}
\begin{aligned}
 &(M+q)b^{K+p}\le c^{M+q}<(M+q+1)b^{K+p}\\
 &\hspace{35mm}\Longleftrightarrow\quad
 \frac{c^M}{b^K}\in J_1.
\end{aligned}
\end{equation}
Consequently,
\begin{equation}\label{eq:floor-resonance-hit-implication}
  \boxed{
  M+q\in\Sset_{c,b}\quad\Longrightarrow\quad M\in\Sset_{c,b}.}
\end{equation}
\end{theorem}

\begin{proof}
If $p<0$, then
\[
  \varrho=c^qb^{-p}\ge2^{q+1},
\]
so $\varrho-1>q$ and $x<1$, contrary to $M\ge1$.  Hence $p\ge0$.

The identity $\varrho=1+q/x$ gives
\[
  \frac{M+q}{\varrho}-M
  =\frac{q(x-M)}{x+q}\ge0
\]
and
\[
  \frac{M+q+1}{\varrho}-(M+1)
  =\frac{q(x-M-1)}{x+q}<0.
\]
The strict inequality in the second display uses the full floor condition
$M\le x<M+1$.  It preserves the half-open upper endpoint and proves
\cref{eq:floor-resonance-containment}; the length is immediate.

Since $c^q=b^p\varrho$, division by $b^{K+p}$ shows that the first line of
\cref{eq:floor-resonance-scale-equivalence} is equivalent to
\[
  M+q\le
  \varrho\,\frac{c^M}{b^K}
  <M+q+1,
\]
which is precisely membership in $J_1$.

Now assume $M+q\in\Sset_{c,b}$ and let $L\ge0$ be a legal radix scale for
that hit:
\[
  (M+q)b^L\le c^{M+q}<(M+q+1)b^L.
\]
Put $K=L-p$.  By \cref{eq:floor-resonance-scale-equivalence},
$c^M/b^K\in J_1\subseteq J_0$.  If $K<0$, then
\[
  \frac{c^M}{b^K}>c^M\ge2^M\ge M+1,
\]
contradicting $J_0=[M,M+1)$.  Thus $K\ge0$, so the same containment is a
legal prefix inequality for $M$.  This proves
\cref{eq:floor-resonance-hit-implication}.
\end{proof}

\begin{remark}[The pair $6,10$]
\label{rem:six-ten-nesting}
For $(c,b)=(2,10)$, take $(p,q)=(1,4)$.  Then
\[
  \varrho=\frac85,\qquad x=\frac{20}3,\qquad M=6,
\]
and
\[
  J_1=\left[\frac{25}4,\frac{55}8\right)\subset[6,7).
\]
The hit at $M=6$ has scale $K=1$, and
$2^6/10=32/5\in J_1$.  Hence the shifted scale is $K+p=2$, recovering the
double hit $6,10\in\Sset_{2,10}$ without an approximate phase calculation.
\end{remark}

\begin{corollary}[The exact absolute-phase subwindow]
\label{cor:floor-resonance-phase-window}
Retain the hypotheses and notation of
\cref{thm:floor-resonance-nesting}, and suppose that
$M\in\Sset_{c,b}$.  Let $K\ge0$ be its unique radix scale and put
\[
  \theta=\{F_{c,b}(M)\},
  \qquad
  \frac{c^M}{b^K}=M b^\theta.
\]
Define
\begin{equation}\label{eq:floor-resonance-phase-endpoints}
  A=\log_b\frac{M+q}{\varrho M},
  \qquad
  B=\log_b\frac{M+q+1}{\varrho M}.
\end{equation}
Then
\begin{equation}\label{eq:floor-resonance-phase-containment}
  0\le A<B<\delta_b(M)
\end{equation}
and
\begin{equation}\label{eq:floor-resonance-phase-equivalence}
  \boxed{
  M+q\in\Sset_{c,b}
  \quad\Longleftrightarrow\quad
  \theta\in[A,B).}
\end{equation}
The subwindow has exact length
\begin{equation}\label{eq:floor-resonance-phase-length}
  B-A=\delta_b(M+q),
\end{equation}
and the phase length removed from the original prefix window is
\begin{equation}\label{eq:floor-resonance-phase-loss}
\begin{aligned}
 \delta_b(M)-\delta_b(M+q)
 &=\log_b\left(
   1+\frac{q}{M(M+q+1)}
 \right)\\
 &<\frac{q}{M(M+q+1)\log b}.
\end{aligned}
\end{equation}
\end{corollary}

\begin{proof}
First note that the radix scale of any hit $n\in\Sset_{c,b}$ is unique.
Indeed, if
\[
  nb^a\le c^n<(n+1)b^a,
\]
then
\[
  b^a\le\frac{c^n}{n}
  <b^a\left(1+\frac1n\right)\le b^{a+1},
\]
so $a=\lfloor\log_b(c^n/n)\rfloor$.

The increasing map $y\mapsto\log_b(y/M)$ sends $J_0$ to
$[0,\delta_b(M))$ and sends $J_1$ to $[A,B)$.  The strict upper endpoint
in the proof of \cref{thm:floor-resonance-nesting} therefore gives
\cref{eq:floor-resonance-phase-containment}.

If $M+q$ is a hit, let $L\ge0$ be its unique scale.  The proof of
\cref{thm:floor-resonance-nesting}, with $K'=L-p$, shows that $K'\ge0$ and
$c^M/b^{K'}\in J_1\subseteq J_0$.  Hence $K'$ is a legal scale for $M$;
uniqueness forces $K'=K$, and therefore $L=K+p$.  This proves the forward
direction of \cref{eq:floor-resonance-phase-equivalence}.  Conversely,
$\theta\in[A,B)$ means $c^M/b^K\in J_1$.  Since $p\ge0$,
\cref{eq:floor-resonance-scale-equivalence} supplies the legal scale
$K+p\ge0$ for $M+q$, proving the reverse direction.

Subtracting the two endpoints in
\cref{eq:floor-resonance-phase-endpoints} gives
\[
  B-A=\log_b\frac{M+q+1}{M+q}=\delta_b(M+q).
\]
Finally,
\[
 \frac{(M+1)(M+q)}{M(M+q+1)}
 =1+\frac{q}{M(M+q+1)},
\]
and $\log(1+t)<t$ proves
\cref{eq:floor-resonance-phase-loss}.
\end{proof}

\begin{theorem}[Finite multistep nesting and endpoint filling]
\label{thm:multistep-floor-resonance}
Retain the hypotheses and notation of
\cref{thm:floor-resonance-nesting}.  For $j\ge0$, put
\[
  n_j=M+jq,
  \qquad
  \mathcal I_j=
  \left[
    \frac{n_j}{\varrho^j},
    \frac{n_j+1}{\varrho^j}
  \right).
\]
For every integer $J\ge1$ one has the exact half-open intersection formula
\begin{equation}\label{eq:multistep-intersection}
  \bigcap_{j=0}^{J}\mathcal I_j
  =
  \begin{cases}
    \displaystyle
    \left[
      \dfrac{M+q}{\varrho},
      \dfrac{M+Jq+1}{\varrho^J}
    \right),
    &\varrho^{J-1}(M+q)<M+Jq+1,\\[3mm]
    \varnothing,
    &\varrho^{J-1}(M+q)\ge M+Jq+1.
  \end{cases}
\end{equation}
Consequently,
\begin{equation}\label{eq:multistep-Jmax}
  J_{\max}=
  \max\left\{
    J\ge1:\varrho^{J-1}(M+q)<M+Jq+1
  \right\}
\end{equation}
is well defined and finite.

For every $K\in\mathbb Z$ and $j\ge0$ there is also the exact equivalence
\begin{equation}\label{eq:multistep-scale-equivalence}
\begin{aligned}
 &(M+jq)b^{K+jp}\le c^{M+jq}
   <(M+jq+1)b^{K+jp}\\
 &\hspace{34mm}\Longleftrightarrow\quad
 \frac{c^M}{b^K}\in\mathcal I_j.
\end{aligned}
\end{equation}
If $1\le J\le J_{\max}$ and both $M+q$ and $M+Jq$ belong to
$\Sset_{c,b}$, then
\begin{equation}\label{eq:multistep-endpoint-filling}
  \boxed{
    M,M+q,M+2q,\ldots,M+Jq\in\Sset_{c,b}.}
\end{equation}
Moreover, their unique radix scales form an arithmetic progression
\[
  K,K+p,K+2p,\ldots,K+Jp
\]
for some $K\ge0$.
\end{theorem}

\begin{proof}
Write $t=x-M\in[0,1)$.  Since $\varrho-1=q/x$, the lower and upper
endpoints
\[
  \ell_j=\frac{M+jq}{\varrho^j},
  \qquad
  u_j=\frac{M+jq+1}{\varrho^j}
\]
satisfy
\begin{align*}
  \ell_{j+1}-\ell_j
  &=\frac{(\varrho-1)(t-jq)}{\varrho^{j+1}},\\
  u_{j+1}-u_j
  &=\frac{(\varrho-1)(t-1-jq)}{\varrho^{j+1}}.
\end{align*}
Thus $\ell_1\ge\ell_0$, the sequence $(\ell_j)_{j\ge1}$ is strictly
decreasing, and $(u_j)_{j\ge0}$ is strictly decreasing.  The largest lower
endpoint among $\mathcal I_0,\ldots,\mathcal I_J$ is therefore $\ell_1$,
whereas the smallest upper endpoint is $u_J$.  This proves
\cref{eq:multistep-intersection}, including the strict inequality forced by
the half-open upper endpoint.  Since $u_J\to0$ while $\ell_1>0$,
\cref{eq:multistep-Jmax} is finite.

The identity $c^q=b^p\varrho$ gives
\[
  \frac{c^{M+jq}}{b^{K+jp}}
  =\varrho^j\frac{c^M}{b^K},
\]
which proves \cref{eq:multistep-scale-equivalence}.  The preceding nesting
theorem already shows that $p\ge0$.

It remains to prove endpoint filling.  The case $J=1$ is the preceding
nesting theorem, so assume $2\le J\le J_{\max}$ and set
\[
  n_1=M+q,
  \qquad n_J=M+Jq.
\]
The strict decrease of the lower endpoints after $j=1$, together with the
strict nonemptiness condition in \cref{eq:multistep-Jmax}, gives
\begin{equation}\label{eq:multistep-raw-lower}
  n_J<\varrho^{J-1}n_1<n_J+1.
\end{equation}
Let $L_1$ be the unique radix scale of the hit $n_1$, and put
\[
  z=\frac{c^{n_1}}{b^{L_1}}\in[n_1,n_1+1),
  \qquad R=\varrho^{J-1}z.
\]
Then $R>n_J$, while
\begin{equation}\label{eq:multistep-raw-upper}
  R
  <(n_J+1)\left(1+\frac1{n_1}\right)
  \le2n_J.
\end{equation}
Here $n_1\ge2$ and $n_J\ge3$, and the last inequality is equivalent to
$(n_J-1)(n_1-1)\ge2$.

Let $L_J$ be the unique radix scale of the hit $n_J$, and define
\[
  d=L_J-\bigl(L_1+(J-1)p\bigr)\in\mathbb Z.
\]
Then
\[
  \frac{c^{n_J}}{b^{L_J}}=\frac{R}{b^d}\in[n_J,n_J+1).
\]
If $d\le-1$, the left side is at least $bR>2n_J\ge n_J+1$; if
$d\ge1$, \cref{eq:multistep-raw-upper} makes it smaller than
$R/b<n_J$.  Hence $d=0$: the two endpoint scales are automatically
coherent.

By the preceding nesting theorem, $K=L_1-p\ge0$ is the unique scale of
$M$.  With $y=c^M/b^K=z/\varrho$, the two coherent endpoint hits say that
$y\in\mathcal I_1\cap\mathcal I_J$.  The endpoint monotonicity proved
above gives
\[
  \mathcal I_1\cap\mathcal I_J\subseteq\mathcal I_j
  \qquad(0\le j\le J).
\]
Now \cref{eq:multistep-scale-equivalence} gives every hit in
\cref{eq:multistep-endpoint-filling}; all scales $K+jp$ are nonnegative
because $K,p\ge0$.
\end{proof}

\begin{remark}[Size and necessity of the coherent range]
\label{rem:multistep-range}
Put $s=q(\varrho-1)$.  Along any family of resonance data for which
$s\to0$, the exact threshold in \cref{eq:multistep-Jmax} satisfies
\begin{equation}\label{eq:multistep-Jmax-asymptotic}
  J_{\max}\sim\sqrt{\frac{2}{q(\varrho-1)}}.
\end{equation}
Indeed, put $h=\varrho-1$, $s=qh$, $f=x-M\in[0,1)$, and $k=J-1$.
After substituting $M=q/h-f$ and multiplying by $h/q$, the strict
nonemptiness condition is exactly
\[
  (1+h)^k\left(1+h-\frac{fh}{q}\right)
  <1+(k+1)h+\frac{(1-f)h}{q}.
\]
If $G_k$ denotes the left side minus the right side, then, uniformly for
$k=O(s^{-1/2})$,
\[
  G_k=-\frac hq+
  \left(\frac{k(k+1)}2-\frac{kf}{q}\right)h^2+O(k^3h^3),
\]
and therefore
\[
  \frac qhG_k=-1+\frac{sk^2}{2}+o(1).
\]
Here $kh=o(1)$, $sk=o(1)$, and the normalized remainder is
$qk^3h^2=(sk^2)(kh)=o(1)$.  Comparing
$k=(1\pm\eta)\sqrt{2/s}$ and then letting $\eta\downarrow0$ locates the
last strict inequality.  Since the upper endpoints decrease with $J$, the
condition cannot recover after it fails, proving the asymptotic.

For positive-error continued-fraction convergents $p_k/q_k$, this becomes
\[
  J_{\max,k}
  \sim\sqrt{\frac{2q_{k+1}}{(\log b)q_k}}
\]
only along subsequences for which $q_{k+1}/q_k\to\infty$; without this
large-next-partial-quotient hypothesis one must retain the exact inequality
in \cref{eq:multistep-Jmax}.

The condition $J\le J_{\max}$ cannot be dropped from endpoint filling.
For $(c,b)=(2,10)$ and $(p,q)=(1,4)$ one has $M=6$ and
$J_{\max}=1$.  Both $10=M+q$ and $1542=M+384q$ are hits, whereas
$14=M+2q$ is not.  Thus two separated hits need not belong to the same
coherent scale chain.  Finally, finiteness of $J_{\max}$ for each fixed
resonance is a search reduction, not a proof of finiteness or infinitude of
$\Sset_{2,10}$.
\end{remark}

\begin{proposition}[Exact coherent depth and a quadratic certificate]
\label{prop:exact-coherent-depth}
Retain the floor-resonance data in
\cref{thm:floor-resonance-nesting,thm:multistep-floor-resonance}.  Put
\[
  \gamma=\log\varrho,
  \qquad B=M+q,
  \qquad A=M+q+1,
\]
and define
\begin{equation}\label{eq:coherent-depth-lambert-root}
  \kappa=
  -\frac1\gamma
  W_{-1}\left(
    -\frac{\gamma B}{q}\e^{-\gamma A/q}
  \right)-\frac Aq.
\end{equation}
Then $\kappa>0$ and the coherent depth has the exact closed form
\begin{equation}\label{eq:coherent-depth-exact}
  \boxed{J_{\max}=\lceil\kappa\rceil.}
\end{equation}
If $k=J_{\max}-1$, then the following strictly necessary inequality holds:
\begin{equation}\label{eq:coherent-depth-quadratic-certificate}
  \boxed{kq(qk+q-2)<2(M+1).}
\end{equation}
Consequently,
\begin{equation}\label{eq:coherent-depth-quadratic-bound}
  \boxed{
  J_{\max}<
  1+\frac{\sqrt{(q-2)^2+8(M+1)}-(q-2)}{2q}}
\end{equation}
and, more simply,
\begin{equation}\label{eq:coherent-depth-simple-bound}
  J_{\max}<
  \frac{3+\sqrt{1+8(M+1)/q^2}}2
  <2+\frac{\sqrt{2(M+1)}}q.
\end{equation}
\end{proposition}

\begin{proof}
Writing the real variable $u=J-1$, the strict nonemptiness condition in
\cref{eq:multistep-Jmax} is
\[
  G(u):=B\e^{\gamma u}-(A+qu)<0.
\]
Here $G(0)=-1$, $G(-A/q)>0$, $G(u)\to\infty$ as $u\to\infty$, and
$G''(u)>0$.  Thus $G$ has exactly one negative and one positive zero; if
$\kappa$ is the positive one, then $G(u)<0$ for $0\le u<\kappa$.
Putting $y=u+A/q$ in $G(u)=0$ gives
\[
  (-\gamma y)\e^{-\gamma y}
  =-\frac{\gamma B}{q}\e^{-\gamma A/q}.
\]
The larger real solution is on the branch $W_{-1}$ and is exactly
\cref{eq:coherent-depth-lambert-root}.  Therefore the admissible integers
satisfy $J-1<\kappa$.  If $\kappa$ is an integer, equality is excluded by
the strict upper endpoint; in both the integral and nonintegral cases the
largest admissible $J$ is $\lceil\kappa\rceil$.  This proves
\cref{eq:coherent-depth-exact} without an endpoint ambiguity.

For the certificate, write $h=\varrho-1=q/x$ and
$x=M+t$ with $0\le t<1$.  At $k=J_{\max}-1$ one has
\[
  B\bigl((1+h)^k-1\bigr)<kq+1.
\]
The binomial lower bound through degree two yields
\begin{equation}\label{eq:coherent-depth-binomial-step}
  k(Bh-q)+\binom{k}{2}Bh^2<1.
\end{equation}
The floor relation gives
\[
  Bh-q=\frac{q(q-t)}x>
  \frac{q(q-1)}{M+1},
  \qquad
  Bh^2=\frac{q^2B}{x^2}>
  \frac{q^2}{M+1}.
\]
Substitution in \cref{eq:coherent-depth-binomial-step} proves
\cref{eq:coherent-depth-quadratic-certificate}.  Solving its quadratic
inequality for $k$ gives \cref{eq:coherent-depth-quadratic-bound}.  If one
keeps only the quadratic binomial term, then
$k(k-1)q^2<2(M+1)$, which gives the first simpler bound in
\cref{eq:coherent-depth-simple-bound}; the last inequality is elementary.
\end{proof}

\begin{corollary}[Summability of convergent floor centers]
\label{cor:floor-resonance-center-summability}
Assume that $c,b\ge2$ are multiplicatively independent integers.  Let
$p_j/q_j$ be the regular continued-fraction convergents of
$\alpha=\log_b c$, and restrict to the indices for which
\[
  \varepsilon_j=q_j\alpha-p_j>0.
\]
Write $q_{j+1}$ for the denominator immediately following $q_j$ in the
full convergent sequence, and set
\[
  x_j=\frac{q_j}{b^{\varepsilon_j}-1},
  \qquad
  M_j=\lfloor x_j\rfloor.
\]
Then, along these positive-error indices and for all sufficiently large
$j$,
\begin{equation}\label{eq:floor-resonance-center-size}
  \boxed{M_j\asymp_b q_jq_{j+1}.}
\end{equation}
In particular,
\begin{equation}\label{eq:floor-resonance-center-sum}
  \boxed{
  \sum_{\substack{j\ge0\\\varepsilon_j>0}}
  \delta_b(M_j)<\infty,}
\end{equation}
where finitely many initial indices with $M_j<1$ are omitted.
\end{corollary}

\begin{proof}
The standard two-sided convergent estimate gives
\[
  \frac1{q_j+q_{j+1}}
  <\varepsilon_j
  <\frac1{q_{j+1}}.
\]
For $0<t\le1$, convexity and $\e^u\ge1+u$ give
\[
  (\log b)t\le b^t-1\le(b-1)t.
\]
Consequently,
\[
  \frac{q_jq_{j+1}}{b-1}
  <x_j
  <\frac{q_j(q_j+q_{j+1})}{\log b}
  \le\frac{2q_jq_{j+1}}{\log b}.
\]
The quantities $x_j$ tend to infinity, so taking the floor changes them
by at most a fixed multiplicative factor for all sufficiently large $j$.
This proves \cref{eq:floor-resonance-center-size}.  Moreover,
\[
  \delta_b(M_j)
  \le\frac1{M_j\log b}
  \ll_b\frac1{q_jq_{j+1}}.
\]
Convergent denominators grow at least at the Fibonacci rate, hence the
last majorant is summable even before restricting to positive-error
indices.  This proves \cref{eq:floor-resonance-center-sum}.
\end{proof}

\begin{theorem}[Lacunarity and total complexity of the coherent skeleton]
\label{thm:convergent-center-skeleton}
Retain the positive-error convergent centers of
\cref{cor:floor-resonance-center-summability}, omitting the finite initial
set for which $M_j<1$.  Let $a_n$ denote the regular continued-fraction
partial quotients.  Consecutive positive-error indices are $j$ and $j+2$,
and they satisfy
\begin{equation}\label{eq:center-strong-lacunarity}
  \boxed{
  M_{j+2}\ge(a_{j+2}+1)^2M_j\ge4M_j,
  \qquad q_{j+2}\ge2q_j.}
\end{equation}
Hence, if
\[
  C_+(X)=\#\{j:\varepsilon_j>0, 1\le M_j\le X\},
\]
then
\begin{align}
  C_+(X)&\le1+\lfloor\log_4X\rfloor,
    \label{eq:center-count-lacunary}\\
  C_+(X)&\le\frac{\log(X+1)}{4\log\phi}+O_b(1),
    \qquad \phi=\frac{1+\sqrt5}{2},
    \label{eq:center-count-fibonacci}
\end{align}
and, for every $\sigma>0$,
\begin{equation}\label{eq:center-all-power-sums}
  \sum_{\substack{j\ge0\\\varepsilon_j>0,\ M_j\ge1}}
  M_j^{-\sigma}<\infty.
\end{equation}

Let $J_j$ be the exact value $J_{\max}$ attached to the center $M_j$, and
let
\[
  \mathcal K_j=\{M_j+r q_j:0\le r\le J_j\}
\]
be its full potential coherent skeleton.  Then
\begin{equation}\label{eq:weighted-coherent-depth-sum}
  \boxed{
  \sum_{\substack{j\ge0\\\varepsilon_j>0,\ M_j\ge1}}
  \frac{J_j+1}{\sqrt{M_j}}<\infty.}
\end{equation}
In fact the left-hand side is at most $10$.  Consequently,
\begin{equation}\label{eq:coherent-skeleton-little-o}
  T(X):=
  \sum_{\substack{\varepsilon_j>0\\M_j\le X}}(J_j+1)
  =o_{c,b}(\sqrt X),
\end{equation}
with the explicit uniform estimate
\begin{equation}\label{eq:coherent-skeleton-explicit-bound}
  T(X)<
  3(1+\lfloor\log_4X\rfloor)+2\sqrt{2(X+1)}.
\end{equation}
If $\mu(\alpha)<\infty$, then for every $\epsilon>0$,
\begin{equation}\label{eq:coherent-skeleton-finite-type}
  \boxed{
  T(X)\ll_{c,b,\epsilon}
  X^{1/2-1/\mu(\alpha)+\epsilon}.}
\end{equation}
Thus $\mu(\alpha)=2$ gives $T(X)\ll_{c,b,\epsilon}X^\epsilon$; if
$\alpha$ has bounded partial quotients, then
$T(X)\ll_{c,b}\log(2+X)$.
For every fixed multiplicatively independent integer pair $(c,b)$,
linear-form estimates give an effective finite upper bound for
$\mu(\alpha)$.  Hence there is an effective, possibly very small,
$\eta_{c,b}>0$ such that
\begin{equation}\label{eq:coherent-skeleton-effective-saving}
  T(X)\ll_{c,b}X^{1/2-\eta_{c,b}}.
\end{equation}

Finally, the intervals spanned by $\mathcal K_j$ are pairwise disjoint for
all sufficiently large positive-error indices.
\end{theorem}

\begin{proof}
Let $\Delta_n=|q_n\alpha-p_n|$.  The continued-fraction recurrences and
alternation of the errors give
\[
  q_{j+2}=a_{j+2}q_{j+1}+q_j,
  \qquad
  \Delta_j=a_{j+2}\Delta_{j+1}+\Delta_{j+2}.
\]
Since $q_{j+1}\ge q_j$ and $\Delta_{j+1}>\Delta_{j+2}$, the first ratio used
below is at least $a_{j+2}+1$, whereas the second is strictly greater than
$a_{j+2}+1$.  Also
$\psi(t)=(b^t-1)/t$ is strictly increasing for $t>0$.  Therefore, for the
unfloored centers $x_j=q_j/(b^{\Delta_j}-1)$,
\[
  \frac{q_{j+2}}{q_j}\ge a_{j+2}+1,
  \qquad
  \frac{b^{\Delta_j}-1}{b^{\Delta_{j+2}}-1}
  =\frac{\Delta_j}{\Delta_{j+2}}
   \frac{\psi(\Delta_j)}{\psi(\Delta_{j+2})}
  >a_{j+2}+1,
\]
and hence
\[
  \frac{x_{j+2}}{x_j}
  =\frac{q_{j+2}}{q_j}
   \frac{b^{\Delta_j}-1}{b^{\Delta_{j+2}}-1}
  >(a_{j+2}+1)^2.
\]
The multiplier is an integer, so taking floors preserves the weak
inequality in \cref{eq:center-strong-lacunarity}.  The denominator
inequality follows from the first recurrence.  Iteration proves
\cref{eq:center-count-lacunary,eq:center-all-power-sums}.  It also gives the
local bound
\[
  \#\{j:Y<M_j\le X,\ \varepsilon_j>0\}
  \le\left\lceil\log_4\frac XY\right\rceil
  \qquad(0<Y<X).
\]
For \cref{eq:center-count-fibonacci}, use
$M_j+1>x_j>q_jq_{j+1}/(b-1)$ and the Fibonacci lower growth of convergent
denominators.  Since the positive errors occupy one parity, the product
$q_jq_{j+1}$ grows at least like $\phi^{4r}$ along the $r$th such index.

By \cref{eq:coherent-depth-simple-bound},
\[
  J_j+1<3+\frac{\sqrt{2(M_j+1)}}{q_j}.
\]
The lacunarity in \cref{eq:center-strong-lacunarity} gives
\[
  \sum M_j^{-1/2}\le2,
  \qquad
  \sum q_j^{-1}\le2.
\]
Since $\sqrt{(M_j+1)/M_j}\le\sqrt2$, division by $\sqrt{M_j}$ proves
\cref{eq:weighted-coherent-depth-sum} with upper bound $6+4=10$.
If $u_j=(J_j+1)/\sqrt{M_j}$, then $(u_j)\in\ell^1$ and
\[
  \frac{T(X)}{\sqrt X}
  =\sum_{M_j\le X}u_j\sqrt{\frac{M_j}{X}}\longrightarrow0
\]
by a finite-head, small-tail argument.  This proves
\cref{eq:coherent-skeleton-little-o}; using the two displayed geometric
sums before taking the limit proves
\cref{eq:coherent-skeleton-explicit-bound}.

For the finite-type refinement, write $Q=q_{j+1}$ and $R=Q/q_j$.
The center-size estimates give
\begin{equation}\label{eq:center-ratio-two-constraints}
  J_j+1\ll_b1+\sqrt R,
  \qquad q_j^2R\ll_bX
  \quad(M_j\le X).
\end{equation}
For every $\nu>\mu(\alpha)$, all sufficiently large convergents also
satisfy $R\ll_{\alpha,\nu}q_j^{\nu-2}$.  Split the positive-error
denominators at $Y=(X+1)^{1/\nu}$.  Since they at least double, the part
$q_j\le Y$ contributes $O(Y^{(\nu-2)/2})$, while
\cref{eq:center-ratio-two-constraints} makes the part $q_j>Y$ contribute
\[
  O\left(\sqrt X\sum_{q_j>Y}\frac1{q_j}\right)
  =O\left(\frac{\sqrt X}{Y}\right).
\]
Both quantities are $O(X^{1/2-1/\nu})$; the $O(\log X)$ contribution of
the constant term is absorbed.  Letting $\nu\downarrow\mu(\alpha)$ and
absorbing the remaining logarithm proves
\cref{eq:coherent-skeleton-finite-type}.  Bounded partial quotients make
$R$ bounded, so \cref{eq:center-count-lacunary} gives the stated logarithmic
bound.

It remains to prove eventual disjointness.  The upper convergent-error bound
gives
\[
  q_j<\sqrt{(b-1)(M_j+1)}.
\]
Together with \cref{eq:coherent-depth-simple-bound}, this yields
\[
  M_j+J_jq_j
  <M_j+\bigl(2\sqrt{b-1}+\sqrt2\bigr)\sqrt{M_j+1}.
\]
For all sufficiently large $M_j$ the last expression is less than $4M_j$,
which is at most the next positive-error center by
\cref{eq:center-strong-lacunarity}.  Hence the skeleton intervals are
eventually disjoint.
\end{proof}

\begin{remark}[Uniform sharpness of the square-root scale]
For fixed $b$ and $c=b^p+1$, the first positive convergent is $p/1$ and its
floor center is exactly $M=b^p$.  Here
$\varrho-1=b^{-p}$, so
\cref{eq:multistep-Jmax-asymptotic} gives
$J_{\max}\sim\sqrt{2M}$ as $p\to\infty$.  Thus the square-root scale in the
uniform estimate \cref{eq:coherent-skeleton-explicit-bound} cannot be
improved uniformly over integer parameter pairs.  This example, like the
whole theorem, concerns the potential coherent skeleton and does not assert
that any of its points are hits.
\end{remark}

\begin{remark}
\label{rem:floor-resonance-boundary}
The nesting theorem is a one-way structural implication.  It does not show
that any floor resonance center is a hit, and
\cref{eq:floor-resonance-center-sum} is not a deterministic
Borel--Cantelli theorem for the prescribed phases.  Even a proof that only
finitely many such centers produce double hits would not decide whether
$\Sset_{2,10}$ is finite or infinite, because general solutions need not
occur at these centers.
\end{remark}

%% file: sections/en/06_certified_search.tex
\section{Certified continued-fraction localization and complexity}

We now convert the separation of an irrational rotation into a block search
that is both exact and output-sensitive.  The first lemma records the range
in which an apparent real interval may be unfolded without wrapping around
the circle.

\paragraph{Computational model.}
In every bit-complexity statement, $c$ and $b$ (and $\nu$, when it occurs)
are fixed, whereas the search bound $N$ is given in binary.  We use the
deterministic bit-operation model.  Integers of magnitude $N^{O(1)}$ have
$O(\log N)$ bits, and division, gcd, and Euclidean floor-sum operations are
charged by their bit costs.  Certified logarithms at precision $p$ are
evaluated by deterministic ball algorithms of $p^{O(1)}$ bit cost; standard
background on multiprecision arithmetic and midpoint-radius ball arithmetic
is given in~\cite{brentzimmermann2010,johansson2017arb}.  We use
$\operatorname{polylog}N$ for
$(\log N)^{O_{c,b,\nu}(1)}$.  The phrase ``arithmetic steps'' below counts
multiprecision integer operations, whereas the bit-complexity statement in
\cref{thm:interpolated-blocks} includes operand sizes, certified
transcendental evaluations, and exact final verification.  After a fixed
parameter-dependent initial range is checked directly, every reported index
is tested against both half-open prefix endpoints.  Thus the final output,
as opposed to an internal locator report, is exactly
$\Sset_{c,b}\cap[1,N]$.

\begin{lemma}[Exact phase bins]\label{lem:phase-bins}
Assume $c\ge2$ and write $\alpha=s+\rho$ with
$s\in\mathbb Z_{\ge0}$ and $0<\rho<1$.  Given $M,H\ge1$, set
\[
  x_j=\{(M+j)\alpha-\log_b(M+1)\}\quad(1\le j\le H)
\]
and
\[
  t_j=\log_b\frac{M+j+1}{M+1}\quad(0\le j\le H).
\]
If $t_H<1$, then for every $1\le j\le H$,
\begin{equation}\label{eq:phase-bin}
  \boxed{M+j\in\Sset_{c,b}\Longleftrightarrow
  x_j\in[t_{j-1},t_j).}
\end{equation}
\end{lemma}

\begin{proof}
We have
\[
  \{F_{c,b}(M+j)\}=\{x_j-t_{j-1}\},
  \qquad
  \delta_b(M+j)=t_j-t_{j-1}.
\]
If $x_j<t_{j-1}$, then
$\{x_j-t_{j-1}\}=1+x_j-t_{j-1}\ge1-t_{j-1}>t_j-t_{j-1}$,
where the final inequality uses $t_j<1$.  If $x_j\ge t_{j-1}$, the target
inequality is exactly $x_j<t_j$.  This proves
\cref{eq:phase-bin}.
\end{proof}

\begin{theorem}[Packing in convergent blocks]\label{thm:convergent-packing}
Retain the notation and hypotheses of \cref{lem:phase-bins}, and assume that
$\rho$ is irrational.  Let $q_{k-1}<q_k$, $q_k\ge2$, be
adjacent denominators of regular continued-fraction convergents of $\rho$,
and put
\[
  d_k=\|q_{k-1}\rho\|.
\]
If $H\le q_k$ and
\begin{equation}\label{eq:single-threshold}
  M\ge\frac{q_k}{b^{d_k}-1}-1,
\end{equation}
then $(M,M+H]$ contains at most one solution.
\end{theorem}

\begin{proof}
The best-approximation property gives
\begin{equation}\label{eq:best-separation}
  \min_{1\le h<q_k}\|h\rho\|=d_k,
  \qquad
  \frac1{q_k+q_{k-1}}<d_k<\frac1{q_k}.
\end{equation}
Under \cref{eq:single-threshold},
\[
  t_H\le t_{q_k}
  =\log_b\left(1+\frac{q_k}{M+1}\right)\le d_k<\frac12.
\]
Every solution phase lies in $[0,t_H)$ by \cref{lem:phase-bins}.  Two
distinct phases differ on the circle by
$\|(i-j)\rho\|\ge d_k$, while two points in the half-open arc
$[0,t_H)$ have distance strictly less than $t_H\le d_k$.  This is
impossible.
\end{proof}

The next theorem uses a factor $1/2$ to provide a fixed reserve under a
simple uniform grid bound.  This lets a fixed-point outer approximation
report at most one \emph{candidate} without scanning the block.  It is not
an intrinsic limit: the packing form below works whenever $t_H<d_k$ if the
grid is refined according to the remaining margin.

\begin{theorem}[Safe convergent-block locator]\label{thm:safe-locator}
Retain the notation of \cref{thm:convergent-packing}.  Suppose
\begin{equation}\label{eq:safe-locator-hypotheses}
  H\le q_k,
  \qquad t_H\le\frac{d_k}{2}.
\end{equation}
A convenient sufficient condition for
\cref{eq:safe-locator-hypotheses} is
\begin{equation}\label{eq:safe-threshold}
  H\le q_k,
  \qquad
  M\ge\frac{q_k}{b^{d_k/2}-1}-1.
\end{equation}
Let
\[
  \theta=\{M\alpha-\log_b(M+1)\},
  \qquad x_j=\{\theta+j\rho\}.
\]
Choose an integer $Q$ and take $B,P$ to be certified nearest-integer
roundings of $Q\theta$ and $Q\rho$, respectively.  Thus, with ties settled
by either fixed convention, they satisfy
\begin{equation}\label{eq:fixed-approximations}
  \dist_{\T}\left(\theta,\frac BQ\right)\le\frac1{2Q},
  \qquad
  \left|\rho-\frac PQ\right|\le\frac1{2Q},
\end{equation}
and
\begin{equation}\label{eq:Q-safe}
  Q>16(H+3)(q_k+q_{k-1}).
\end{equation}
Using certified interval arithmetic, round the two endpoints of
\[
  \left[-\frac{H+1}{Q},
  t_H+\frac{H+1}{Q}\right]\pmod1
\]
outward to the $Q^{-1}$ grid, with at most $2/Q$ additional expansion past
the certified enclosure, and hence less than $3/Q$ relative to the true
endpoint.  Denote by $\mathcal J_Q\subset\T$ the resulting union of at
most two circular intervals with endpoints on the $Q^{-1}$ grid.  Then
\begin{equation}\label{eq:report-set}
  \mathcal R(M,H)=\left\{1\le j\le H:
    \frac{(B+Pj)\bmod Q}{Q}\in\mathcal J_Q\right\}
\end{equation}
contains every genuine solution in $(M,M+H]$ and has cardinality at most
one.  It can be reported in
\[
  O(\log H\log Q)
\]
multi-precision integer arithmetic steps by Euclidean floor sums and binary
search.
\end{theorem}

\begin{proof}
Here ``outward'' has the following exact algorithmic meaning.  Split the
circular arc at zero if necessary.  For each lifted component
$[a,b]\subset\mathbb R$, compute certified enclosures
$a\in[a^-,a^+]$ and $b\in[b^-,b^+]$, increasing the precision until both
enclosure widths are less than $Q^{-1}$.  Replace that component by
\[
  \left[\frac{\lfloor Qa^-\rfloor-1}{Q},
        \frac{\lceil Qb^+\rceil+1}{Q}\right]
\]
and reduce modulo one.  The true arc is contained in the result.  The
enclosure width, directed rounding, and the one guard cell together enlarge
either true endpoint by less than $3/Q$.  Thus this prescription is also
strict pseudocode for constructing $\mathcal J_Q$ and makes no unrecorded
floating-point rounding assumption.

The hypotheses give $t_H\le d_k/2<1$, so
\cref{lem:phase-bins} applies.  Put
\[
  y_j=\left\{\frac{B+Pj}{Q}\right\},
  \qquad e=\frac{H+1}{Q},
  \qquad \eta=e+\frac2Q=\frac{H+3}{Q}.
\]
By \cref{eq:fixed-approximations},
$\dist_\T(x_j,y_j)\le e/2$ for every $j\le H$.  Hence a genuine hit, for
which $x_j\in[0,t_H)$, is never omitted by the outward-rounded arc.

Conversely, a reported index satisfies
\[
  y_j\in[-e-3/Q,t_H+e+3/Q]\pmod1,
\]
and therefore
\[
  x_j\in[-3e/2-3/Q,t_H+3e/2+3/Q]\pmod1
  \subset[-2\eta,t_H+2\eta]\pmod1.
\]
By \cref{eq:Q-safe,eq:best-separation}, $\eta<d_k/16$.  The last arc has
length
\[
  t_H+4\eta<\frac{d_k}{2}+\frac{d_k}{4}<d_k.
\]
Two reported indices would have exact phases separated by at least $d_k$
according to \cref{eq:best-separation}, a contradiction.

It remains to locate the possible residue without scanning.  For
$0\le U\le Q$,
\begin{align}\label{eq:floor-sum-count}
  \#\{0\le j<n:(B+Pj)\bmod Q<U\}
  ={}&n+\sum_{j<n}\left\lfloor\frac{Pj+B}{Q}\right\rfloor\notag\\
  &-\sum_{j<n}\left\lfloor
    \frac{Pj+B+Q-U}{Q}\right\rfloor.
\end{align}
Each sum in \cref{eq:floor-sum-count} is evaluated by the Euclidean
floor-sum algorithm in $O(\log Q)$ arithmetic steps.  A circular window is
the union of at most two residue intervals.  Count the full index interval,
then descend only into the unique nonempty half.  Since there is at most one
reported index, the binary-search depth is $O(\log H)$.  To count indices
$1\le j\le H$ using the displayed $0\le j'<n$ convention, replace
$j$ by $j'+1$ and $B$ by $B+P$ modulo $Q$.
\end{proof}

\begin{proposition}[Packing and output-sensitive reporting]
\label{prop:output-sensitive-locator}
Retain the construction of \cref{thm:safe-locator}, but assume only
$H\le q_k$ and $t_H<1$.  Put
\[
  \eta=\frac{H+3}{Q},
  \qquad L_Q=t_H+4\eta.
\]
If $L_Q<1$, then every genuine solution is still reported and
\begin{equation}\label{eq:output-sensitive-packing}
  \boxed{
  \#\mathcal R(M,H)
  \le1+\left\lfloor\frac{L_Q}{d_k}\right\rfloor.}
\end{equation}
In particular, if $t_H<d_k$ and
\begin{equation}\label{eq:margin-grid-condition}
  Q>\frac{4(H+3)}{d_k-t_H},
\end{equation}
then at most one candidate is reported.  If
$R=\#\mathcal R(M,H)$, all reported indices can be located in
\begin{equation}\label{eq:output-sensitive-floor-sum-complexity}
  O\left(
  \bigl[1+R(1+\log(2H/\max\{1,R\}))\bigr]\log Q
  \right)
\end{equation}
multi-precision integer arithmetic steps.
\end{proposition}

\begin{proof}
The corrected enclosure calculation in the preceding proof uses neither
$t_H\le d_k/2$ nor \cref{eq:Q-safe}: it shows that every reported exact
phase lies in one circular arc of length $L_Q$, while genuine solutions are
never omitted.  Distinct exact phases with indices in $[1,H]$ are separated
by at least $d_k$, because every nonzero index difference is less than
$q_k$.  Lifting the arc to an interval of length $L_Q<1$ and ordering its
points gives \cref{eq:output-sensitive-packing}.  Condition
\cref{eq:margin-grid-condition} makes $L_Q<d_k$, proving the singleton
claim.

For the complexity statement, evaluate the floor-sum count on a binary
partition of the index interval and descend only into nonempty children.
At each depth the number of nonempty parents is at most the smaller of the
number of dyadic cells and $R$, and at most twice as many children are
tested.  Summing over the depths gives
$O(1+R(1+\log(2H/\max\{1,R\})))$ interval counts, each costing
$O(\log Q)$ arithmetic steps by \cref{eq:floor-sum-count}.
\end{proof}

\begin{lemma}[Polylogarithmic endpoint certification]
\label{lem:polylog-certification}
Fix multiplicatively independent integers $c,b\ge2$ and a constant $C\ge1$.
There is an effective constant $K=K(c,b,C)$ with the following property.
Let $N\ge3$, let $1\le r,s\le2N+1$, and let
$u_1,u_2,u_3,u_4\in\mathbb Z$ satisfy
$|u_i|\le N^{C+1}$.  After omitting terms whose logarithm is zero, every
nonzero linear form
\begin{equation}\label{eq:endpoint-linear-form}
  \Lambda
  =u_1\log c+u_2\log b+u_3\log r+u_4\log s
\end{equation}
satisfies
\begin{equation}\label{eq:endpoint-separation}
  |\Lambda|\ge
  \exp\bigl(-K(\log N)^3\bigr).
\end{equation}
Consequently, if $M,H,m\le N$ and $Q\le N^C$, then all of the following
can be certified with $O_{c,b,C}((\log N)^3)$ bits of ball precision:
\begin{enumerate}[label=\textup{(\alph*)},leftmargin=2.2em]
  \item the continued-fraction comparisons needed to construct the
  convergents of $\rho=\{\log_b c\}$ with denominator at most $N^C$;
  \item the floors, ceilings, and directed grid endpoints used for
  $\theta=\{M\log_b c-\log_b(M+1)\}$ and
  $t_H=\log_b((M+H+1)/(M+1))$ in \cref{thm:safe-locator};
  \item the two final comparisons
  $mb^k\le c^m<(m+1)b^k$ for a reported index.
\end{enumerate}
If one of the corresponding forms is exactly zero, equality can instead be
decided in time polynomial in $\log N$ by a gcd-free basis computation on
the finitely many integer arguments.  Thus every endpoint construction and
final verification used by the locator has deterministic
$\operatorname{polylog}N$ bit complexity.
\end{lemma}

\begin{proof}
Matveev's theorem for a nonzero linear form in logarithms of rational
integers gives a lower bound whose logarithm is, up to a constant depending
only on the number of terms, the product of the logarithmic heights of the
arguments and $1+\log\max_i|u_i|$.  The two fixed arguments $c,b$ have
constant height, while $r,s\le2N+1$ have height $O(\log N)$ and
$\log\max_i|u_i|=O_C(\log N)$.  Hence the exponent is
$O_{c,b,C}((\log N)^3)$, proving
\cref{eq:endpoint-separation}.  This deliberately uniform cubic bound also
covers forms with only one variable integer argument, for which a smaller
power of $\log N$ would suffice.

We next reduce the listed comparisons to
\cref{eq:endpoint-linear-form}.  Comparing a rational $p/q$ with
$\log_b c$ amounts, after multiplication by $\log b$, to the sign of
$q\log c-p\log b$.  To round $Q\theta$ or determine its integer part, one
compares
\[
  QM\log c-Q\log(M+1)-v\log b
\]
with zero for an appropriate integer $v$.  Likewise, after multiplication
by $Q\log b$, a grid comparison involving $t_H$ and the rational error
budget reduces to
\[
  Q\log(M+H+1)-Q\log(M+1)-w\log b
\]
for an integer $w$; the terms $H+1$ or $H+3$ in the error budget are
absorbed into $w$.  Finally, the two prefix endpoints are the signs of
\[
  m\log c-k\log b-\log m,
  \qquad
  m\log c-k\log b-\log(m+1).
\]
All coefficients are $O_{c,b}(N^{C+1})$ under the displayed hypotheses.
Enlarging the fixed exponent parameter by one absorbs this constant for all
sufficiently large $N$; the remaining finite range is handled directly.
Once the evaluation radius is below half the
right-hand side of \cref{eq:endpoint-separation}, a nonzero interval cannot
straddle zero, proving the precision claim.  Standard algorithms evaluate
the required logarithms to that precision in time polynomial in the bit
  length.  There are $O(\log N)$ continued-fraction steps, together with
  $O((\log N)^2)$ multiprecision integer steps for one singleton floor-sum
  localization as in \cref{thm:safe-locator}; every step uses only a fixed
  number of arithmetic or logarithm evaluations on
$O_{c,b,C}((\log N)^3)$-bit balls.  Thus the precision bound indeed yields
the stated total polylogarithmic bit complexity, rather than merely a
per-evaluation precision estimate.

For completeness, exact equality does not require factoring the variable
integers.  Repeatedly split the fixed list $c,b,r,s$ by pairwise gcds until
a gcd-free basis $g_1,\ldots,g_v$ is obtained, and write each original
integer as a product of powers of these pairwise coprime bases.  Then
$\Lambda=0$ is equivalent to the vanishing of the exponent sum at every
$g_i$.  The list has bounded length, and gcd-free refinement uses
  polynomially many gcd and exact-division operations on $O(\log N)$-bit
  integers; see also~\cite{bachshallit1996}.  This proves the equality and
  complexity assertions.
\end{proof}

The fixed-threshold strategy can be sharpened by interpolating between two
successive safe denominators.  This removes the loss caused by waiting until
the next denominator is safe for a full block.

\begin{theorem}[Interpolated convergent blocks]
\label{thm:interpolated-blocks}
\label{cor:adaptive-complexity}
Let $p_k/q_k$ be the convergents of the irrational number $\rho$, after
discarding the possible initial repetition of the denominator $1$, and put
\begin{equation}\label{eq:interpolated-level-data}
  d_k=\|q_{k-1}\rho\|,
  \qquad
  \lambda_k=b^{d_k/2}-1,
  \qquad
  T_k=\frac{q_k}{\lambda_k}-1.
\end{equation}
The thresholds $T_k$ are strictly increasing and tend to infinity.  At a current left endpoint
$M$ in the range
\begin{equation}\label{eq:interpolated-active-level}
  T_k\le M<T_{k+1},
\end{equation}
set
\begin{equation}\label{eq:interpolated-block-length}
  A=(M+1)\lambda_{k+1},
  \qquad
  H_0=\max\{q_k,\lfloor A\rfloor\}.
\end{equation}
For each convergent denominator define its half-margin admissible length
\begin{equation}\label{eq:half-margin-admissible-length}
  H_j(M)=\min\{q_j,\lfloor(M+1)\lambda_j\rfloor\}.
\end{equation}
Whenever it is positive, $H_j(M)$ satisfies
\cref{eq:safe-locator-hypotheses} with the $j$th adjacent pair.
Then the interpolated length is the exact upper envelope
\begin{equation}\label{eq:interpolated-optimal-envelope}
  \boxed{H_0=\max_{j\ge1}H_j(M).}
\end{equation}
If $\lfloor A\rfloor\le q_k$, apply \cref{thm:safe-locator} with the
adjacent denominator pair $(q_{k-1},q_k)$; otherwise apply it with
$(q_k,q_{k+1})$.  In either case $(M,M+H_0]$ is reduced to at most one
candidate.  A final block may be truncated to
$H=\min\{H_0,N-M\}$.

If $\mu(\rho)<\infty$, then for every $\nu>\mu(\rho)$ this interpolated
search uses
\begin{equation}\label{eq:interpolated-block-count}
  O_{\rho,b,\nu}\left(N^{1-1/\nu}\right)
\end{equation}
blocks up to $N$.  For fixed multiplicatively independent integers $c,b$,
including certified construction of the convergents, the block length and
grid endpoints, Euclidean floor sums, and exact final verification, the
total bit complexity is
\begin{equation}\label{eq:adaptive-time}
  \boxed{
  O_{c,b,\nu}\left(
  N^{1-1/\nu}\operatorname{polylog}N\right).}
\end{equation}
If $\rho$ has bounded partial quotients, the exponent in
\cref{eq:adaptive-time} is $1/2$.
\end{theorem}

\begin{proof}
The errors $d_k$ strictly decrease to zero while the denominators $q_k$
strictly increase to infinity.  Hence $\lambda_k$ strictly decreases to
zero, while $T_k$ strictly increases to infinity.  Moreover, by
\cref{eq:best-separation},
\begin{equation}\label{eq:lambda-denominator-size}
  \lambda_k\asymp_b\frac1{q_k}.
\end{equation}

We first verify the two locator branches, including all floor effects.  If
$\lfloor A\rfloor\le q_k$, then $H_0=q_k$.  The lower inequality in
\cref{eq:interpolated-active-level} gives
\[
  \frac{H_0}{M+1}=\frac{q_k}{M+1}\le\lambda_k,
\]
and therefore
\[
  t_{H_0}=\log_b\left(1+\frac{H_0}{M+1}\right)
  \le\log_b(1+\lambda_k)=\frac{d_k}{2}.
\]
Thus \cref{eq:safe-locator-hypotheses} holds for
$(q_{k-1},q_k)$.  If $\lfloor A\rfloor>q_k$, then $H_0=\lfloor A\rfloor$.
The upper inequality in \cref{eq:interpolated-active-level} yields
\[
  q_k<H_0<q_{k+1},
  \qquad
  \frac{H_0}{M+1}\le\lambda_{k+1}.
\]
Hence $H_0\le q_{k+1}$ and
\[
  t_{H_0}\le\log_b(1+\lambda_{k+1})=\frac{d_{k+1}}2,
\]
which is \cref{eq:safe-locator-hypotheses} for $(q_k,q_{k+1})$.  In the
relevant branch choose a grid modulus
\[
  Q>16(H_0+3)(q_j+q_{j-1}),
\]
where $q_j$ is the larger denominator being used.  The one-candidate
assertion follows from \cref{thm:safe-locator}; truncation only strengthens
its hypotheses.

For every positive $H_j(M)$, its definition gives
$H_j(M)\le q_j$ and $H_j(M)/(M+1)\le\lambda_j$, hence
$t_{H_j(M)}\le d_j/2$.  Thus it is half-margin admissible.  The same
inequalities prove the optimal-envelope statement.  Indeed,
$M\ge T_k$ gives $H_k(M)=q_k$.  For $j\le k$ one has
$H_j(M)\le q_j\le q_k$, whereas for $j\ge k+1$ the monotonicity of
$\lambda_j$ gives
\[
  H_j(M)\le\lfloor(M+1)\lambda_j\rfloor
  \le\lfloor A\rfloor.
\]
Finally, $A<q_{k+1}$ gives
$H_{k+1}(M)=\lfloor A\rfloor$.  Thus the maximum over all convergents is
exactly \cref{eq:interpolated-block-length}, proving
\cref{eq:interpolated-optimal-envelope}.

Fix $\nu>\mu(\rho)$.  Adjacent denominators satisfy
\begin{equation}\label{eq:adjacent-denominator-growth}
  q_{k+1}\ll_{\rho,\nu}q_k^{\nu-1}.
\end{equation}
In the first branch, $A<q_k+1$, so by
\cref{eq:lambda-denominator-size},
\[
  M+1\ll_b q_kq_{k+1}
  \ll_{\rho,b,\nu}q_k^\nu.
\]
Thus $H_0=q_k\gg M^{1/\nu}$.  In the second branch, $A>q_k$ and hence
$M+1\gg_bq_kq_{k+1}$.  Combining this with
\cref{eq:adjacent-denominator-growth} gives
\[
  q_{k+1}\ll_{\rho,b,\nu}M^{(\nu-1)/\nu}.
\]
Also $A>q_k\ge2$ beyond a finite initial range, so
$H_0=\lfloor A\rfloor\ge A/2$.  Consequently,
\[
  H_0\gg_b\frac{M}{q_{k+1}}
  \gg_{\rho,b,\nu}M^{1/\nu}.
\]
Every full block starting in a dyadic interval $[X,2X)$ therefore has
length $\gg X^{1/\nu}$.  Summing the resulting
$O(X^{1-1/\nu})$ bound over dyadic intervals proves
\cref{eq:interpolated-block-count}.

It remains to justify certification rather than only real-arithmetic
operation counts.  From $T_k\le M\le N$ and
\cref{eq:lambda-denominator-size}, one has $q_k\ll_bN^{1/2}$; then
\cref{eq:adjacent-denominator-growth} gives
$q_{k+1}\ll_{\rho,\nu}N^{(\nu-1)/2}$.  Thus the required grid modulus may
be chosen with $Q\le N^{O_\nu(1)}$.  Choosing the active index $k$ amounts
to certifying the two threshold comparisons in
\cref{eq:interpolated-active-level}.  The comparison defining
$H_0=\lfloor(M+1)(b^{d_{k+1}/2}-1)\rfloor$ is, after taking logarithms,
a sign comparison of
\[
  2\log\frac{M+r+1}{M+1}-d_{k+1}\log b
\]
for an appropriate integer $r$.  Since $d_{k+1}\log b$ is the absolute
value of $q_k\log c-v\log b$ for an integer $v$, this is a linear form in
the logarithms of $c,b,M+1,M+r+1$.  For a full block $r\le N-M$; for the
last truncated block it suffices to compare $A$ with $N-M$ rather than
construct a larger $H_0$.  More generally, the active denominators and all
arguments and coefficients in the floor and threshold comparisons are
bounded by $N^{O_\nu(1)}$.  More explicitly, one may apply
\cref{lem:polylog-certification} with ambient parameter
$N'=\max\{N,q_{k+1}\}$ and enlarge it by a fixed power to include $Q$.
Since $\log N'=O_\nu(\log N)$, this certifies the choice of $k$ and $H_0$,
as well as every convergent, phase, grid-endpoint, floor-sum, and final
prefix comparison, with polylogarithmic bit cost per block.  A large partial
quotient is processed by one Euclidean integer-quotient operation, not by
repeated subtraction.  Multiplication by
\cref{eq:interpolated-block-count} proves \cref{eq:adaptive-time}.

When the partial quotients are bounded, $q_{k+1}\ll q_k$.  Both branches
then give $H_0\gg M^{1/2}$, yielding exponent $1/2$.
\end{proof}

\begin{corollary}[Exact-output certified search]
\label{cor:exact-output-certified-search}
Fix multiplicatively independent integers $c,b\ge2$, suppose
$\mu(\{\log_b c\})<\infty$, and let
$\nu>\mu(\{\log_b c\})$.  There is a deterministic certified algorithm
which, on binary input $N$, returns exactly
$\Sset_{c,b}\cap[1,N]$ in
\[
  O_{c,b,\nu}\left(
  N^{1-1/\nu}\operatorname{polylog}N\right)
\]
bit operations.  The rational-grid locator may internally report a strict
superset, but each reported index is subjected to the original inequalities
in \cref{eq:def-S}; unverified reports are never included in the output.
\end{corollary}

\begin{proof}
Apply \cref{thm:interpolated-blocks} after directly checking its finite
parameter-dependent initial range.  The locator never omits a genuine hit,
and \cref{lem:polylog-certification} certifies the final endpoint tests.  The
complexity is already included in \cref{eq:adaptive-time}.
\end{proof}

\begin{corollary}[Certified constants for $2^m$]\label{cor:two-safe-block}
For $\rho=\log_{10}2$, two adjacent convergents are
\[
  \frac{1\,936\,274}{6\,432\,163},
  \qquad
  \frac{13\,456\,039}{44\,699\,994},
\]
and
\[
  d=\|6\,432\,163\log_{10}2\|
  =2.0276549588410307779\ldots\times10^{-8}.
\]
The least integer satisfying the safe threshold is
\begin{equation}\label{eq:two-safe-threshold}
  \boxed{M_{\rm safe}=1\,914\,818\,931\,502\,442.}
\end{equation}
Thus, for every $M\ge M_{\rm safe}$ and every
$H\le44\,699\,994$, the entire block $(M,M+H]$ is reduced by
\cref{thm:safe-locator} to at most one candidate for exact verification.
The convergents,
the strict inequality at \cref{eq:two-safe-threshold}, and failure at the
preceding integer were certified with $1024$-bit ball arithmetic.
\end{corollary}

\begin{remark}
The computational supplementary archive contains a standalone verifier and
its machine-readable output:
\begin{center}
  \small
  \nolinkurl{supplement/verify_safe_singleton_threshold.py}\\
  \nolinkurl{supplement/safe_singleton_threshold_certificate.json}.
\end{center}
The
dependency is pinned as \nolinkurl{python-flint==0.6.0}; the verifier has
separate SHA-256 digests for the script, canonical JSON, and requirements
file in the supplementary README.  That README also records the tested
CPython environment and a
\texttt{--verify-certificate} command that recomputes every Arb comparison
and requires the saved JSON to agree byte for byte.

The fixed denominator in \cref{cor:two-safe-block} reduces the number of
high-precision final checks by a factor on the order of $4.47\times10^7$,
but it does not by itself produce a sublinear global algorithm.  Sublinear
complexity requires increasing convergent denominators, and none of these
upper bounds proves that a block is nonempty.
\end{remark}

%% file: sections/en/07_fixed_parameter_problems.tex
\section{Open problems at the critical scale}

The exact status of \cref{eq:core-case} can now be summarized without
heuristic shortcuts.

\begin{problem}[Infinitude for $2^m$]\label{prob:two-infinite}
Does \cref{eq:two-exact-one-sided} hold for infinitely many $j$?  Equivalently,
is $\Sset_{2,10}$ infinite?
\end{problem}

\begin{problem}[The logarithmic law]
Is
\[
  A_{2,10}(X)\sim\log_{10}X?
\]
By \cref{thm:exact-discrepancy}, this is equivalent to the signed
discrepancy estimate $D_{\log_{10}2,10}(X)=o(\log X)$.
\end{problem}

\begin{problem}[Critical one-sided approximation]
Determine the sign pattern of
\[
  d_+(x_j)-(z_j-x_j),
\]
where both endpoints are given exactly by \cref{eq:xz-lambert}.  The
all-orders expansion \cref{eq:lagrange-width} makes the target explicit;
the missing input is a deterministic one-sided approximation theorem for
the prescribed nonlinear sequence $x_j$ at scale $1/j$.
\end{problem}

The power-saving discrepancy theorem and its subcritical moving-target
consequence, \cref{prop:two-log-discrepancy,thm:two-subcritical-moving},
the structural results above, and every finite certified computation remain
compatible with either finiteness or infinitude.  They stop strictly before
the critical scale $1/j$.  This logical boundary is the main obstacle rather
than a numerical inability to generate further candidates.  In particular,
the certified locator should be read as a rigorous reduction of the search
space, not as evidence for a proof by exhaustion.

%% file: sections/en/08_submission_declarations.tex
\section*{Code and data availability}

The standalone verifier, machine-readable certificate, pinned dependency,
cryptographic checksums, and reproducibility instructions supporting the
certified computation in Corollary~\ref{cor:two-safe-block} are publicly
available in the versioned repository \cite{fang2026selfprefixcode}, release
v1.0.0, and are also provided as supplementary material with this article.

\section*{Funding}

This research did not receive any specific grant from funding agencies in the
public, commercial, or not-for-profit sectors.

\section*{Declaration of competing interest}

The author declares no competing interests.

\section*{CRediT authorship contribution statement}

\textbf{Zihang Fang:} Conceptualization, Formal analysis, Investigation,
Methodology, Software, Validation, Writing--original draft,
Writing--review and editing.

\section*{Declaration of generative AI and AI-assisted technologies in the
manuscript preparation process}

During the preparation of this work, the author used ChatGPT and Codex
(OpenAI) to assist with literature searches and proof-checking.  The author
independently checked all mathematical arguments, references, and
computations, reviewed and edited the resulting material as needed, and takes
full responsibility for the content of the published article.